\documentclass[letter,english]{article}

\usepackage[utf8]{inputenc}
\usepackage[T1]{fontenc}
\usepackage{babel}
\usepackage[babel]{csquotes}
\usepackage{xcolor}

\usepackage{enumerate}

\usepackage{amsmath}
\usepackage{amsfonts}
\usepackage{amsthm}
\usepackage{bbm}
\usepackage{amssymb}
\usepackage{url}

\usepackage{float}
\usepackage{tikz}
\usetikzlibrary{shapes.misc}
\tikzset{cross/.style={cross out, draw, 
         minimum size=2*(#1-\pgflinewidth), 
         inner sep=0pt, outer sep=0pt}}
\usepackage{IEEEtrantools}

\usepackage[pdftex,colorlinks]{hyperref}

\newtheorem{theorem}{Theorem}
\newtheorem{proposition}{Proposition}
\newtheorem{lemma}{Lemma}
\newtheorem{corollary}{Corollary}
\theoremstyle{definition}
\newtheorem{definition}{Definition}

\newtheorem*{remark}{Remark}

\DeclareMathOperator{\covol}{covol}
\DeclareMathOperator{\Vol}{Vol}

\DeclareMathOperator{\End}{End}

\DeclareMathOperator{\rk}{rk}
\DeclareMathOperator{\Span}{Span}
\DeclareMathOperator{\Hom}{Hom}
\DeclareMathOperator{\Grass}{Grass}
\DeclareMathOperator{\Supp}{Supp}
\DeclareMathOperator{\Sym}{Sym}
\DeclareMathOperator{\diag}{diag}

\newcommand{\trace}{\mathrm{Tr}}
\newcommand{\GL}{\mathrm{GL}}
\newcommand{\SL}{\mathrm{SL}}
\newcommand{\Lie}{\mathrm{Lie}}

\newcommand{\C}{\mathbb{C}}
\newcommand{\R}{\mathbb{R}}
\newcommand{\Z}{\mathbb{Z}}
\newcommand{\Q}{\mathbb{Q}}
\newcommand{\QQ}{\overline{\mathbb{Q}}}
\newcommand{\N}{\mathbb{N}}

\newcommand{\g}{\mathfrak{g}}

\newcommand{\bw}{\mathbf{w}}

\newcommand{\bx}{\mathbf{x}}

\newcommand{\bv}{\mathbf{v}}
\newcommand{\bu}{\mathbf{u}}
\newcommand{\bp}{\mathbf{p}}
\newcommand{\bq}{\mathbf{q}}
\newcommand{\bc}{\mathbf{c}}
\newcommand{\bd}{\mathbf{d}}

\newcommand{\cO}{\mathcal{O}}
\newcommand{\cL}{\mathcal{L}}
\newcommand{\cF}{\mathcal{F}}
\newcommand{\cP}{\mathcal{P}}
\newcommand{\cB}{\mathcal{B}}

\newcommand{\sr}{\mathsf{r}}
\newcommand{\sx}{\mathsf{x}}

\newcommand{\bg}{\mathbf{g}}
\newcommand{\bh}{\mathbf{h}}
\newcommand{\bk}{\mathbf{k}}

\newcommand{\eps}{\varepsilon}

\title{A subspace theorem for manifolds}
\author{Emmanuel Breuillard and Nicolas de Saxcé}

\begin{document}

\maketitle

\begin{abstract}
We prove a theorem that generalizes Schmidt's Subspace Theorem in the context of metric diophantine approximation. To do so we reformulate the Subspace theorem in the framework of homogeneous dynamics by introducing and studying a slope formalism and the corresponding notion of semistability for diagonal flows. 
\end{abstract}

\section*{Introduction}

In 1972, Wolfgang Schmidt formulated his celebrated \emph{subspace theorem} \cite[Lemma~7]{schmidt_nf}, a far reaching generalization of results of Thue \cite{thue}, Siegel \cite{siegel}, and Roth \cite{roth} on rational approximations to algebraic numbers.
Around the same time, in his work on arithmeticity of lattices in Lie groups, Gregory Margulis \cite{margulis1971} used the geometry of numbers to establish the recurrence of unipotent flows on the space of lattices $\GL_d(\R)/\GL_d(\Z)$. More than two decades later, a quantitative refinement of this fact, the so-called quantitative non-divergence estimate, was used by Kleinbock and Margulis \cite{km} in their solution to the Sprindzuk conjecture regarding the extremality of non-degenerate manifolds in metric diophantine approximation.
As it turns out, these two remarkable results -- the subspace theorem and the Sprindzuk conjecture -- are closely related and can be understood together as statements about diagonal orbits in the space of lattices. In this paper we prove a theorem that generalizes both results at the same time. We also provide several applications.   This marriage is possible thanks to a better understanding of the geometry lying behind the subspace theorem, in particular the notion of Harder-Narasimhan filtration for one-parameter diagonal actions, which leads both to a dynamical reformulation of the original subspace theorem and to a further geometric understanding of the family of exceptional subspaces arising in Schmidt's theorem. The proof blends the diophantine input of Schmidt's original theorem with the dynamical input arising from the Kleinbock-Margulis approach and refinements recently obtained in \cite{abrs2} and \cite{dfsu}.


\bigskip

We now formulate the main theorem. Let $M=\phi(U)$ be a connected analytic submanifold of $\GL_d(\R)$ parametrized by an analytic map $\phi: U \to \GL_d(\R)$, where $U \subset \R^n$ is a connected open set and $n \in \N$. Let $\mu$ be the push-forward in $M$ of the Lebesgue measure on $U$. The Zariski closure of $M$ is said to be defined over $\QQ$ if for $(a_{ij})_{ij} \in \GL_d$  the ideal of polynomial functions in $\C[a_{ij}, \det^{-1}]$ that vanish on $M$ can be generated by polynomials with coefficients in $\QQ$.

\begin{theorem}[Subspace theorem for manifolds] 
\label{subspace-main}
Assume that the Zariski closure of $M$ in $\GL_d$ is defined over $\QQ$. 
Then there exists a finite family of proper subspaces $V_1,\dots,V_r$ of $\Q^d$ with $r=r(d)$ such that, for $\mu$-almost every $L$ in $M$, for every $\eps>0$, the integer solutions $\bx\in\Z^d$ to the inequality
\begin{equation}\label{total} \prod_{i=1}^d |L_i(\bx)| \leq \|\bx\|^{-\eps},\end{equation}
all lie in the union $V_1\cup\dots\cup V_r$, except a finite number of them.
\end{theorem}

Here the $L_i$ are the linear forms on $\R^d$ given by the rows of $L \in \GL_d(\R)$, and $\|\cdot\|$ is the canonical Euclidean norm on $\R^d$.

We will in fact prove the theorem under a slightly weaker assumption on $M$ requiring only that what we call the \emph{Pl\"ucker closure} of $M$ be defined over $\QQ$, see \S \ref{schubert-sec} for the definition of  Pl\"ucker closure.

Note that we recover the original Schmidt's subspace theorem \cite{schmidt_da, edixhoven-evertse, bombierigubler} as the special case when $M$ is a singleton $\{L\}$ (then $n=0$ and $\mu$ is the Dirac mass at $L$). On the other hand, even in the case where $M$ is defined over $\Q$, the theorem is non-trivial.
Indeed, as we shall see in Section \ref{sprin}, it recovers the main result of Kleinbock and Margulis regarding the Sprindzuk conjecture.


The exceptional subspaces $V_i$ are independent of $\eps$, as in Vojta's refinement \cite{vojta, schmidt-refinement}, and they depend only the $\emph{rational}$ Zariski closure of $M$. In fact they are determined by what we call the \emph{rational Schubert closure} of $M$, that is the intersection of all rational translates $\mathcal{S}_\sigma g:=\overline{B\sigma B}$ containing $M$, where $g \in \GL_d(\Q)$ and $\mathcal{S}_\sigma$ is a standard Schubert variety associated to a permutation $\sigma$ and a Borel subgroup $B$ containing the diagonal subgroup.  Each $V_i$ contains infinitely many solutions to $(\ref{total})$, regardless of $\eps$. The number $r$ of exceptional subspaces can be bounded by a number depending only on $d$ (see Lemma \ref{values} and the remark following it).

The proof of Theorem \ref{subspace-main} goes via the proof of a stronger result, a parametric subspace theorem for manifolds, Theorem \ref {parametrici} below. This reformulates the problem in terms of the dynamics of a one-parameter diagonal flow $(a_t)_{t>0}$ on the space of lattices. We may summarize it informally as follows:

\begin{theorem}[Parametric version]\label{thm-2} For $\mu$-almost every $L$ in $M$ the lattice $a_tL\Z^d$ assumes a fixed asymptotic shape as $t$ tends to $+\infty$.
\end{theorem}

By ``fixed asymptotic shape'' we mean two things. Firstly that the successive minima are asymptotic to $e^{\Lambda_k t}$ for some real numbers $\Lambda_k$, Lyapunov exponents of sorts, depending only on $M$ and $a=(a_t)_{t>0}$ (the dependence on $a$ is piecewise linear). In particular as $t$ varies there can only be oscillations of subexponential size for successive minima. And secondly that the successive minima determine a fixed partial flag in $\Z^d$.
In other words there is a fixed partial rational flag $W_1\leq \ldots \leq W_d$ in $\Q^d$, such that if $\Lambda_k<\Lambda_{k+1}$, then the $k$ first successive minima of $a_tL\Z^d$ are always realized by vectors from $W_k$ when $t$ is large enough. The $\Lambda_k$'s and the flag $\{W_k\}_k$ depend only on $a$ and on the rational Schubert closure of $M$. Grouping together the different $W_i$ obtained by varying the one-parameter subgroup $a$ we obtain the family of exceptional subspaces $V_i$ appearing in Theorem \ref{subspace-main}.

This rational flag arises naturally as the \emph{Harder-Narasimhan filtration} associated to a certain submodular function on the rational grassmannian: the maximal expansion rate of the subspace under the flow. We recall in \S \ref{h-n-sec} that any submodular function on a grassmannian gives rise to a Grayson polygon, a notion of semistability, a Harder-Narasimhan filtration and certain coefficients, the slopes of the polygon, which in our case will correspond to the Lyapunov exponents $\Lambda_k$ mentioned above.  This is the so-called ``slope formalism'', which arises  in particular in the study of Euclidean lattices as first described by Stuhler \cite{stuhler} and Grayson \cite{grayson}, and in many other subjects as well \cite{bost-bourbaki, randriam}.  


Although we have restricted to the current setting for clarity of exposition in this introduction, the result will be proved for more general measures $\mu$ than push-forwards of the Lebesgue measure by analytic maps; the exceptional subspaces then depend only on the Zariski closure of the support of $\mu$. The right technical framework is that of \emph{good measures}, which are closely related to the friendly measures of \cite{kleinbock-lindenstrauss-weiss}, see \S \ref{mgm}.

\bigskip

The paper is organized as follows. In Section \ref{sec-1} we begin by formulating the technical version of Theorem \ref{thm-2} and then proceed to describe the slope formalism on the grassmannian associated to a one-parameter flow and in particular discuss the associated notion of Harder-Narasimhan filtration. The proof of Theorems \ref{subspace-main} and \ref{thm-2} is carried out in \S  \ref{proof-para} and \S \ref{proof-sub} after a discussion of the Kleinbock-Margulis quantitative non-divergence estimates. In Section \ref{section:nf} we formulate and sketch a proof of an extension of Theorem \ref{subspace-main} to arbitrary number fields, which is analogous to the classical extension of Schmidt's subspace theorem due to Schlickewei to multiple places and targets  \cite{schlickewei_padic, bombierigubler, zannier-lectures}. Finally in Section \ref{app} we prove several applications of the main result. 


For the sake of brevity we do not state these applications in the introduction and refer the reader to Section \ref{app} directly instead. Let us only briefly mention that there are five main applications: $(i)$ we explain how to recover the Sprindzuk conjecture (Kleinbock-Margulis theorem) from Theorem \ref{subspace-main}, $(ii)$ we establish a manifold version of the classical Ridout theorem regarding approximation by rationals whose denominators have prescribed prime factors, $(iii)$ we recover the main results of \cite{abrs2} regarding (weighted) diophantine approximation on submanifolds of matrices showing that they hold also for submanifolds defined over $\QQ$ (and not only over $\Q$), $(iv)$ we prove an optimal criterion for strong  extremality (Corollary \ref{criterion}), which answers in this case a question from \cite{kleinbock-margulis-wang,kmb}, $(v)$ we prove a Roth-type theorem for non-commutative diophantine approximation on nilpotent Lie groups, extending to algebraic points what was done for Lebesgue almost every point in our previous work with Aka and Rosenzweig  \cite{abrs,abrs2}.


Further applications and an extension of some of the results of this paper to other reductive groups and homogenous varieties can be found in the second author's forthcoming work \cite{saxce-hdr}.

\section{The main result}\label{sec-1}

\subsection{Dynamical formulation}\label{dyn-sec}

A lattice, that is a discrete subgroup $\Delta$ of rank $d$ in $\R^d$, can be written:
\[ \Delta= \Z u_1\oplus\dots\oplus\Z u_d,\]
where $(u_i)_{1\leq i\leq d}$ is a basis of $\R^d$. And the space $\Omega$ of lattices can be identified with the homogeneous space
\[ \Omega \simeq \GL_d(\R)/\GL_d(\Z).\]
The position of a lattice $\Delta$ in the space $\Omega$, up to a bounded error, is described by its successive minima $\lambda_1(\Delta)\leq\dots\leq\lambda_d(\Delta)$, defined by
\[ \lambda_i(\Delta)
= \inf\{\lambda>0\ |\ \rk(\Delta\cap B(0,\lambda))\geq i\},\] where $\rk(A)$ for $A \subset \Delta$ denotes the rank of the free abelian subgroup of $\Delta$ generated by $A$, and $B(0,\lambda)$ is the Euclidean ball of radius $\lambda$ centered at the origin in $\R^d$. 

Theorem \ref{subspace-main} will be deduced from the following description of the asymptotic behavior of the successive minima along a diagonal orbit of the lattice $L\Z^d$, where $L$ is a $\mu$-generic point of $M$. Here, as in Theorem \ref{subspace-main}, $M=\phi(U)$ is the image of a connected open set in some Euclidean space $U \subset \R^n$ under an analytic map $\phi: U \to \GL_d(\R)$, and $\mu$ is the push-forward under $\phi$ of the Lebesgue measure on $U$. 


\begin{theorem}[Strong parametric subspace theorem for manifolds]
\label{parametrici} Assume that the Zariski closure of $M$ is defined over $\QQ$. Let $(a_t)_{t \ge 0}$ be a diagonal one-parameter semigroup in $\GL_d(\R)$. Then there exist real numbers $\Lambda_1\leq\Lambda_2\leq\dots\leq\Lambda_d$ such that for $\mu$-almost every $L\in M$ and each $k\in\{1,\dots,d\}$, \begin{equation}\label{lim}
 \lim_{t\to +\infty} \frac{1}{t}\log\lambda_k(a_tL\Z^d)=\Lambda_k.\end{equation}
Moreover, if $0=d_0 < d_1<\dots<d_h=d$ in $\{0,\dots,d\}$ are chosen so that
\[ \Lambda_1=\dots=\Lambda_{d_1}<\Lambda_{d_1+1}=\dots=\Lambda_{d_2}<\dots
<\Lambda_{d_{h-1}+1}=\dots=\Lambda_d,\]
then there exist rational subspaces $V_\ell$, $\ell=0,\dots,h$ in $\Q^d$ such that \begin{itemize} \item  $\dim V_\ell=d_\ell$ and  $\{0\}=V_0 < V_1 < \dots < V_{h}=\Q^d$, \item for $\mu$-almost every $L\in M$ the first $d_\ell$ successive minima of $a_tL\Z^d$ are attained in $V_\ell$ provided $t$ is large enough. \end{itemize} In other words: for all $\eps>0$, for $\mu$-almost every $L\in M$, there is $t_{L,\eps}>0$ such that for $t>t_{L,\eps}$, $\ell=1,\ldots,h,$ and $\bx\in\Z^d$,
\begin{equation}\label{limm} \|a_tL\bx\| \leq e^{t(\Lambda_{d_\ell}-\eps)}
\quad\Rightarrow\quad
\bx\in V_{\ell-1}.\end{equation}
\end{theorem}

When $M$ is reduced to a singleton, the above theorem is a refinement of the parametric subspace theorem, often attributed to Faltings and Wüstholz \cite[Theorem~9.1]{faltings-wustholz}. As shown below, it can also be obtained directly from Schmidt's subspace theorem in its original form. The two results are really equivalent.

An important point to make is that $(\ref{lim})$ is a limit and not only a liminf or a limsup. This can be understood as saying that the diagonal orbit of a lattice defined over $\QQ$ has an \emph{asymptotic shape} at infinity. Of course the $\lambda_k$ can fluctuate, but only up to a small exponential error. This is of course in sharp contrast with what happens for certain specific values of $L$. Indeed it is possible to construct matrices $L$ for which the successive minima have an almost arbitrary behavior along a given diagonal orbit, see \cite[Theorem 1.3]{roy} and \cite[Theorem 2.2]{d-s-u-cras}.


\begin{corollary}[Parametric subspace theorem for manifolds]
\label{paraweak} Keep the same assumptions as in Theorem \ref{parametrici}. Let $a=(a_t)_{t \ge0}$ be a one-parameter diagonal semigroup in $\SL_d(\R)$. There exists a proper subspace $V(a)$ of $\Q^d$ such that given $\eps>0$, for $\mu$-almost every $L\in M$, there is $t_{L,\eps}>0$ such that if  $t>t_{L,\eps}$ and $\bx\in\Z^d$,
\[ \|a_tL\bx\| \leq e^{-\eps t}
\quad\Rightarrow\quad
\bx\in V(a).\]
\end{corollary}
\begin{proof}
Here we have assumed that $a_t$ is unimodular. By Minkowski's theorem the product of all $d$ successive minima of $a_tL\Z^d$ is bounded above and below independently of $t$. In view of $(\ref{lim})$, this implies that $\sum_{i=1}^d\Lambda_i=0.$ So $\Lambda_d \ge 0$. Hence we can take $V(a)=V_{h-1}$ in the notation of Theorem \ref{parametrici}.
\end{proof}

The rational subspaces $V_i$ appearing in Theorem \ref{parametrici} depend only on $(a_t)_{t \ge 0}$ and on the rational Zariski closure of $M$, namely the intersection of the closed algebraic subsets of $\GL_d$ defined over $\Q$ and containing $M$. This will be clear from the proof of Theorem \ref{parametrici} given below, where a more precise description of $V(a)$ and the $V_i$ will be given. As we will see, the filtration $\{V_i\}_i$ is the \emph{Harder-Narasimhan} filtration associated to $M$ and $(a_t)_{t \ge 0}$, and the $\Lambda_i$ are the slopes of the \emph{Grayson polygon}. The next few paragraphs contain preparations towards the proof of Theorem \ref{parametrici} given at the end of this section.

\subsection{Expansion rate and submodularity}\label{exp-rate-sec}

In this subsection $M$ is an arbitrary subset of $\GL_d(\R)$ and $a=(a_t)_{t \ge 0}$ a one-parameter diagonal semigroup. We write $a_t=\diag(e^{A_1t}, \ldots, e^{A_dt})$ for some real numbers $A_1,\ldots,A_d$. For a non-zero subspace $V \leq \R^d$ we define its \emph{expansion rate} with respect to $a$ by 
\begin{equation}\label{taulimit} \tau(V):= \lim_{t \to +\infty} \frac{1}{t} \log \|a_t\bv\|
\end{equation}
where  $\bv$ represents $V$ in an exterior power $\wedge^k\R^d$.  This quantity takes values in the finite set of eigenvalues of $\log a_1$ in exterior powers. More precisely: 
\begin{equation}\label{taulimit2} \tau(V)= \max \{I(a);\ I \subset [1,d], |I|=k, \bv_I \neq 0\}
\end{equation}
where  $I(a)= \sum_{i \in I} A_i$ and $\bv_I$ is the coordinate of $\bv$ in the basis $e_I=e_{i_1} \wedge \ldots \wedge e_{i_k}$ of $\wedge^k \R^d$, where $I=\{i_1,\ldots,i_k\}$, $k=\dim V$.  By convention we will also set $\tau(\{0\})=0$. We leave it to the reader to check that if $A_1\ge \ldots \ge A_d$, then \begin{equation}\label{taulimit2bis}  \tau(V)= I_V(a), \textnormal{ with }
I_V:=\{i \in [1,d], V \cap \mathcal{F}_i > V \cap \mathcal{F}_{i+1} \}
\end{equation}
where $\mathcal{F}_i= \langle e_i,\ldots,e_d \rangle$ and $(e_1,\ldots,e_d)$ is the canonical basis of $\R^d$. 


For a subset $M \subset \GL_d(\R)$, we set: 
\begin{equation}\label{taulimit3} \tau_M(V) := \max_{L \in M} \tau(LV).
\end{equation}
Similarly we see readily that 
\begin{equation}\label{taulimit4} \tau_M(V)= \max \{I(a); I \subset [1,d], |I|=\dim V, (L\bv)_I \neq 0 \textnormal{ for some }L\in M\}.
\end{equation}
From this formula it is clear that for all subspaces $V$ \[ \tau_M(V)=\tau_{Zar(M)}(V)\]
 where $Zar(M)$ is the Zariski closure.


\begin{lemma}[Submodularity of expansion rate]
\label{tausm} Let $M \subset \GL_d(\R)$ and assume that its Zariski closure is irreducible. Then the map $V \mapsto \tau_M(V)$ is submodular on the grassmannian, i.e. satisfies, for every pair of subspaces $W_1,W_2$,
\begin{equation}\label{subm} \tau_M(W_1)+\tau_M(W_2) \geq \tau_M(W_1\cap W_2) + \tau_M(W_1+W_2).\end{equation}
\end{lemma}
\begin{proof}
Given a subspace $W$ in $\R^d$, it is clear from $(\ref{taulimit4})$ that the Zariski closure of  $\{L \in M, \tau_L(W)<\tau_M(W)\}$ is a proper subset of $Zar(M)$.
By irreducibility, we may choose $L$ in $M$ such that $\tau_M=\tau_L$ on all four subspaces $W_1$, $W_2$, $W_1\cap W_2$ and $W_1+W_2$.
It is therefore enough to prove the lemma in the case $M=\{L\}$.
Now let $\bu$ be a vector representing $U=W_1\cap W_2$ in some exterior power of $\R^d$.
Let also $\bw_1'$ and $\bw_2'$ be such that $\bu\wedge\bw_1'$ and $\bu\wedge\bw_2'$ represent $W_1$ and $W_2$, respectively.
The subspace $W_1+W_2$ is then represented by $\bu\wedge\bw_1'\wedge\bw_2'$, and moreover, for every $t$,
\[ \|a_tL(\bu\wedge\bw_1')\|\|a_tL(\bu\wedge\bw_2')\| \geq \|a_tL\bu\|\|a_tL(\bu\wedge\bw_1'\wedge\bw_2')\|.\]
Together with formula \eqref{taulimit}, this shows that $\tau_L$ is submodular. Note that this is compatible with the convention $\tau_M(\{0\})=0$. 
\end{proof}

\subsection{Harder-Narasimhan filtration}\label{h-n-sec}

Submodular functions on partially ordered sets give rise to a ``slope formalism'' as in \cite{stuhler, grayson, faltings-wustholz, bost-bourbaki, randriam}. This is well known. In this paragraph we recall the main facts we need and for the reader's convenience we give short proofs.  The key to them is the following submodularity lemma, which in implicit form goes back at least to Stuhler \cite{stuhler} and Grayson \cite{grayson} in the context of Euclidean sublattices  and their covolume and which we rediscovered in \cite{abrs2} in the present context (subspaces and their expansion rate). Let $k$ be a field and  $\Grass(k^d)$ the Grassmannian of non-zero subspaces of $k^d$. Let  $\phi: \Grass(k^d) \to \R$ be a submodular function, that is satisfying $(\ref{subm})$ with $\phi$ in place of $\tau_M$. 

\begin{lemma}[Submodularity lemma] There is a subspace $V_\phi \in \Grass(k^d)$  achieving the infimum  of $\frac{\phi(V)-\phi(0)}{\dim V}$, $V \in \Grass(k^d)$, and containing all other such subspaces. 
\end{lemma}

\begin{proof}Let $I$ be that infimum. Without loss of generality, up to changing $\phi$ into $\phi-\phi(0)$, we may assume that $\phi(0)=0$. We begin by observing that $\phi$ is bounded below: if $(V_n)_n$ is a sequence of distinct subspaces of maximal dimension with $\phi(V_n) \to -\infty$, pick a fixed line $L$ with $L\not \subset V_n$ for infinitely many $V_n$; by submodularity $\inf_n\phi(V_n+L)=-\infty$, contradicting the maximality of $\dim V_n$.  So $I$ is finite. For $k \ge 1$, we set $I_k$ to be the same infimum restricted to those subspaces $V$ with $\dim V \ge k$. There is a maximal $k_0$ such that $I=I_{k_0}$. If $k_0=d$, then we can take $V_\phi=k^d$ and there is nothing to prove. Otherwise let $\eps>0$ so that $I_{k_0+1}> I + 2\eps$. If a subspace $W$ satisfies
\begin{equation}\tag{$\ast$}\label{star}
\phi(W) \leq (I+\eps) \dim W,
\end{equation}
then $\dim W \leq k_0$. By definition there is  such a subspace with $\dim V_\phi \ge k_0$, call it  $V_\phi$. If $Z$ is another subspace with $\eqref{star}$, then \begin{align*}\phi(Z+V_\phi) &\leq \phi(Z)+\phi(V_\phi) - \phi(Z \cap V_\phi) \\ &\leq  (I+\eps) (\dim Z+\dim V_\phi) - I \dim (Z \cap V_\phi) \\ &\leq  I \dim (Z+ V_\phi) +2\eps \dim (Z + V_\phi) \end{align*} which forces $\dim (Z+V_\phi) \leq k_0$, and hence $Z \leq V_\phi$, as desired.
\end{proof}


\begin{definition}[Semistability] We say that $k^d$ is \emph{semistable} with respect to $\phi$ if  $V_\phi=k^d$.
\end{definition}

\begin{definition}[Grayson polygon] Let $P_\phi:[0,d] \to \R$ be the convex  piecewise linear function that is the supremum of all linear functions whose graph in $[0,d]\times \R$ lies below all points $(\dim V, \phi(V))$, $V \in \Grass(k^d) \cup \{0\}$. Its graph is called the \emph{Grayson polygon} of $\phi$. 
\end{definition}

Let $(d_i,f_i)$, $i=0,\ldots, h$ be the vertices of the Grayson polygon with $d_0=0$ and $d_h=d$, that is the angular points, where the slope changes, i.e.  for $i=1,\ldots h-1$,
$$s_{i} < s_{i+1},$$ where $$s_{i}:=\frac{f_i-f_{i-1}}{d_i-d_{i-1}}.$$

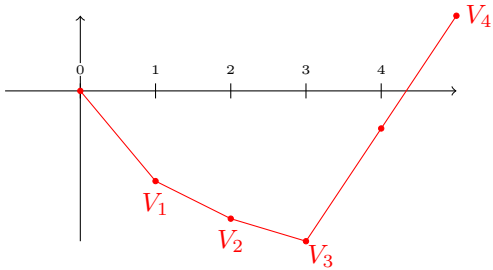
\begin{figure}[h]
\begin{center}
\begin{tikzpicture}
\draw[->] (-1,0) -- (5,0);
\draw (0,-2) -- (0,0.2);
\draw[->] (0,0.37) -- (0,1);
\foreach \x in {0,...,4}
{
\draw (\x,-0.1) -- (\x,0.1) node[anchor=south] {\tiny{\x}};
}
\draw[red,thin] (0,0) -- (1,-1.2) -- (2,-1.7) -- (3,-2) -- (4,-.5) -- (5,1);
\filldraw[red] (0,0) circle (1pt);
\filldraw[red] (1,-1.2) circle (1pt);
\filldraw[red] (1,-1.5) node {$V_1$};
\filldraw[red] (2,-1.7) circle (1pt);
\filldraw[red] (2,-2)  node {$V_2$};
\filldraw[red] (3,-2) circle (1pt) ;
\filldraw[red] (3.2,-2.2)  node {$V_3$};
\filldraw[red] (5,1) circle (1pt);
\filldraw[red] (5.3,1)  node {$V_4$};
\filldraw[red] (4,-.5) circle (1pt);
\end{tikzpicture}
\caption{a Grayson polygon}
\label{hn}
\end{center}
\end{figure}

  The main result is the following:

\begin{proposition}[Harder-Narasimhan filtration]\label{h-n}  For each $i=0,\ldots,h$, there is a unique $k$-subspace $V_i \leq k^d$ such that $\dim V_i=d_i$ and $\phi(V_i)=f_i$. Moreover the subspaces $V_i$ are nested, i.e. $0=V_0<V_1<\dots<V_h=k^d$, forming the so-called \emph{Harder-Narasimhan filtration} of $\phi$. 
\end{proposition}

In particular, we see that $k^d$ is semistable if and only if its Harder-Narasimhan filtration is the trivial one $\{0\} < k^d$.

\begin{remark} \label{hn1} Note that given any $k$-subspace $V \leq k^d$ the function $\phi_V(\overline{W}):= \phi(W)-\phi(V)$ defined on the quotient $k^d/V$, where $\overline{W}=W/V$ for any $k$-subspace $W$ containing $V$ is submodular on $\Grass(k^d/V)$. It is clear from the proposition that  $V_i/V_{i-1}$ is semistable with respect to $\phi_{V_{i-1}}$ for $i=1,\dots,h$, and that $\{V_i/V_1\}_{i \ge 1}$ is the Harder-Narasimhan filtration of $V/V_1$ with respect to $\phi_{V_1}$.
\end{remark}

\begin{proof} The existence and uniqueness of $V_1$ is exactly what the submodularity lemma tells us. Suppose $V_i$  has been defined. We may apply the submodularity lemma again to  $\phi_{V_i}$ on the quotient $k^d/V_i$ and thus obtain a subspace $V_{i+1}$ containing $V_i$ strictly and such that the function $(\phi(V)-\phi(V_i))/(\dim V - \dim V_i)$ reaches  at $V_{i+1}$ its unique minimum $\ell_{i+1}$ among subspaces $V$ containing $V_i$. By construction $\ell_{i}<\ell_{i+1}$ for each $i \ge 1$. 

We now need to show that the Grayson polygon coincides with the polygon $P$ drawn out of the points $(\dim V_i, \phi(V_i))$. In other words we have to prove that if $V$ is a subspace of $k^d$ and $i$ is such that $\dim V_i \leq \dim V < \dim V_{i+1}$, then \begin{equation}\label{toshow}\phi(V)- \phi(V_i) \geq \ell_{i+1} (\dim V - \dim V_i).\end{equation}
If $V_i \leq V$, this is by definition of $V_{i+1}$. Otherwise $V \cap V_i < V_i$ and by induction we may assume that $(\dim (V \cap V_i), \phi(V \cap V_i))$ lies above $P$. So \begin{equation}\label{toshow2}\phi(V_i) - \phi(V \cap V_i) \leq \ell_i(\dim V_i - \dim (V \cap V_i)).\end{equation} Moreover, again by definition of $V_{i+1}$ we have \begin{equation}\label{toshow3} \phi(V_i+V) - \phi(V_i) \geq \ell_{i+1}(\dim(V_i+V) - \dim V_i).\end{equation} On the other hand $\phi$ is submodular, so $\phi(V_i + V) +\phi(V_i \cap V) \leq \phi(V_i)+\phi(V)$. Combining this with $(\ref{toshow2})$ and $(\ref{toshow3})$ we obtain:
\begin{equation}\label{toshow4}\phi(V) \geq \phi(V_i) + \ell_{i+1}(\dim(V) - \dim(V_i \cap V)) - \ell_i(\dim V_i - \dim(V \cap V_i)).\end{equation} But $\ell_{i+1} \ge \ell_i$. So $(\ref{toshow})$ follows. 

This shows the existence of the $V_i$ and the fact that they are nested. To see the uniqueness note that if $\dim V=\dim V_i$ and $\phi(V)=\phi(V_i)$ but $V \neq V_i$, then $\ell_i \ge \ell_{i+1}$ in view of $(\ref{toshow4})$, which is a contradiction.
\end{proof}

In the sequel, we apply this general theory by taking $k=\Q$ and $\phi(V)$ the expansion rate $\tau_M(V)$ defined in $(\ref{taulimit3})$ on the grassmannian of rational subspaces. The above definition of semistability reads:

\begin{definition}\label{semi}
A non-zero rational subspace $V$ in $\R^d$ will be called $M$-\emph{semistable} with respect to $a=(a_t)_{t \ge 0}$ if for every rational subspace $W\leq V$,
\[ \frac{\tau_M(W)}{\dim W} \geq \frac{\tau_M(V)}{\dim V}.\]
\end{definition}

Similarly this yields the notions of Grayson polygon and Harder-Narasimhan filtration of $M$ with respect to $a$. 

\begin{remark}[unstable subspace] In the case when $\det a=1$, a subspace $V$ with $\tau(V)<0$ corresponds to a point $\bv$ in some $\wedge^k \R^d$, which is \emph{unstable with respect to $a$}  in the terminology of geometric invariant theory, i.e. its $a$-orbit contains $0$ in its closure. So $\R^d$ is $M$-semistable if and only there are no unstable subspaces in the full grassmannian.
\end{remark}

Next we make a remark about the dependence of the Harder-Narasimhan filtration with respect to the choice of one-parameter semigroup. It is easy to see that the Grayson polygon depends continuously on $a$. This is not so for the filtration, because new nodes can appear under small deformations, but the following lemma shows that  the subspaces involved remain confined to a fixed finite family. Let $b(n)$ be the ordered Bell number, that is the number of weak orderings (i.e. orderings with ties) on a set with $n$ elements. 

\begin{lemma}\label{values} Let $M \subset \GL_d(\R)$ with irreducible Zariski closure. There is a finite set $S_M$ of rational subspaces of $\R^d$ with $|S_M| \leq b(2^d)$ such that, as $a=(a_t)_{t \ge 0}$ varies among all one-parameter diagonal semigroups of $\GL_d(\R)$, the subspaces $V_i(a)$ arising in the Harder-Narasimhan filtration all  belong to $S_M$. \end{lemma}

\begin{proof}For $I \subset [d]$, let $\widehat{I}(a)=\frac{1}{|I|}\sum_{i \in I} A_i$, where $a_t=\diag(e^{A_1t},\ldots,e^{A_dt})$. We claim that the entire Harder-Narasimhan filtration of $M$ depends on $a$ only via the ordering of the various $\widehat{I}(a)$ for $I \subset [d]$. Namely if $\widehat{I}(a)$ and $\widehat{I}(a')$ define the same weak ordering on the family of subsets of $[d]$, then the filtrations coincide.  To see the claim note that every slope $(\tau_M(V)- \tau_M(W))/(\dim V-\dim W)$ for $W \leq V$ is equal to $\widehat{I}(a)$ for some $I$ (because $I_W \subseteq I_V$ as follows from  $(\ref{taulimit2bis})$ and $\tau_M$ can be replaced by $\tau_L$ for some fixed $L$ as in the proof of Lemma \ref{tausm}), and Proposition \ref{h-n} tells us that  $V_i(a)$ is defined as the unique solution to an extremal problem involving the comparison of slopes. So only their order matters. Since there are at most $b(2^d)$ possible orders, we are done.
\end{proof}

We also see from this proof that the slopes $\Lambda_i$  are  continuous and piecewise linear in $\log a$ and actually linear on each one of the cells cut out by the hyperplanes $\widehat{I}(a)=\widehat{I'}(a)$ for $I,I' \subset [1,d]$.

\begin{remark} The ordered Bell number $b(n)$ grows super-exponentially with $n$. This gives a rather poor bound on the number of exceptional subspaces in Theorem \ref{subspace-main}, especially in view of Schmidt's bound $d^{2d^2}$ from \cite {schmidt-refinement}. A more refined argument, which we do not include here and is based on the study of the set of permutations arising from the Schubert closure of $M$ (see \S \ref{schubert-sec}) allows to improve this to $(2d)^{d}$.  
\end{remark}

\subsection{Dynamics of diagonal flows}
\label{asm}

We now describe the dynamical ingredient of the proof. Using the quantitative non-divergence estimates (see Theorem \ref{klw} below), Kleinbock  showed in \cite{kleinbock_dichotomy} the existence of a well-defined almost sure diophantine exponent for analytic manifolds.  As described in our previous work \cite{abrs2} with  Menny Aka and Lior Rosenzweig  this holds for more general measures and the exponent actually depends only on the Zariski closure of the support of the measure, a property called \emph{inheritance} in this paper, because the measure inherits its exponent from the Zariski closure of its support. We will need the following version of these results:

\begin{theorem}[Inheritance principle]
\label{rate}
Let $(a_t)_{t \ge 0}$ be a one-parameter diagonal semigroup in $\GL_d(\R)$ and for $L \in\GL_d(\R)$, set
\[ \mu_k(L) := \liminf_{t\to+\infty} \sum_{j \leq k} \frac{1}{t} \log\lambda_j(a_tL\Z^d).\]
Let $U \subset \R^n$ a connected open set, $\phi:U \to \GL_d(\R)$ an analytic map, and $\mu$ the image of the Lebesgue measure under $\phi$.  Let $M:=\phi(U)$ and $Zar(M)$ its Zariski closure in $\GL_d(\R)$. Then for $\mu$-almost every $L$ in $M$ and each $k=1,\ldots,d$,
\[ \mu_k(L) = \sup_{L'\in Zar(M)} \mu_k(L').\]
\end{theorem}

\begin{proof}A lemma of Mahler \cite[Theorem 3]{mahler-convex}, which is a simple consequence of Minkowski's second theorem in the geometry of numbers,  asserts that $\lambda_1(\wedge^k g\Z^d)$ is comparable to $\prod_{j \leq k} \lambda_j(g \Z^d)$ within multiplicative constants independent of $g \in \GL_d(\R)$. Since the $k$-th wedge representation $\rho_k:\GL(\R^d) \to \GL(\wedge^k \R^d)$ is an embedding of algebraic varieties, it maps $Zar(M)$ onto $Zar(\rho_k(M))$. These observations allow to reduce the proof to the case where $k=1$, which we now assume. To this end we first recall the quantitative non-divergence estimates in a form established in \cite{kleinbock-anextension}:

\begin{theorem}\label{klw}
\cite[Theorem~2.2]{kleinbock-anextension} Let $M=\phi(U)$ and $\mu$ be as in Theorem \ref{rate}. There are $C,\alpha>0$ such that the following holds. Let $\rho \in (0,1]$ and $t>0$, and let $B:=B(z,r)$ be an open ball such that $B(z,3^nr)$ is contained in $U$. Assume that for each $v_1,\ldots, v_k$ in $\Z^d$ with $\bw:=v_1\wedge v_2\wedge\dots\wedge v_k \neq 0$,  $$\sup_{y \in B } \|a_t\phi(y) \mathbf{w}\| > \rho^k.$$
Then  for every $\eps \in (0,\rho]$, we have:
\[ |\{x \in B ; \lambda_1(a_t\phi(x)\Z^d)\leq \eps \}| \leq C\big(\tfrac{\eps}{\rho}\big)^\alpha |B|. \]
\end{theorem}

Given $\beta \in \R$, we say that a subset  $S \subset \GL_d(\R)$ satisfies the condition (\ref{cbeta})  if the following holds
\begin{equation}\label{cbeta}
\tag{$\mathcal{C}_\beta$}
\begin{array}{c}
\exists c>0:\
\forall k\in\{1,\dots,d\},\ \forall t>0,\ \forall \bw=v_1\wedge\ldots\wedge v_k \in \wedge^k \Z^d\setminus\{0\},\\
\sup_{g \in S} \| a_tg\bw\| \geq c e^{k\beta t}.
\end{array}
\end{equation}
And we set:
\begin{equation}\label{betainf}
\beta(S) := \sup\{\beta\in \R \,|\, S\ \mbox{satisfies}\ (\mathcal{C}_\beta)\}.
\end{equation}
Note that by construction, if $S \subset S'$, then $\beta(S) \leq \beta(S')$.\\

\noindent \underline{1st claim:} For all $L \in S$, $\beta(S) \ge \mu_1(L)$. If $S=\phi(B)$,  where $\phi$ and $B$ are as in Theorem \ref{klw}, equality holds for $\mu$-almost every $L \in \phi(B)$.\\

\noindent \emph{Proof of claim:} If $\beta>\beta(S)$, then (\ref{cbeta}) fails. This implies that there exists $t$ arbitrarily large such that $\sup_{g \in S} \| a_tg\bw\| \leq  e^{k\beta t}$ for some $\bw \neq 0$. However by Minkowski's theorem applied to the sublattice represented by $a_tg\bw$, this means that $\sup_{g \in S} \lambda_1( a_tg\Z^d) \leq  e^{\beta t}$. Hence $\mu_1(L) \leq \beta$.

The opposite inequality for $S=\phi(B)$ will follow from the quantitative non-divergence estimate combined with Borel-Cantelli. Let $\beta<\beta(\phi(B))$. Then  (\ref{cbeta}) holds and, given $\delta>0$, Theorem \ref{klw} applies with $\rho:= c e^{\beta t}$ and $\eps=e^{(\beta-\delta)t}$ so that 
\[ |\{x \in B ; \lambda_1(a_t\phi(x)\Z^d)\leq e^{(\beta-\delta)t} \}| \leq C e^{-\alpha \delta t}  |B|. \]
Summing this over all $t \in \N$, we obtain by Borel-Cantelli that for almost every $x \in B$, $\lambda_1(a_t\phi(x)\Z^d)\geq e^{(\beta-\delta)t}$ if $t$ is a large enough integer. But this clearly implies that $\lambda_1(a_t\phi(x)\Z^d)\geq e^{(\beta-\delta/2)t}$ for all large enough $t>0$. Hence $\mu_1(\phi(x)) \ge \beta-\delta/2$ for Lebesgue almost every $x \in B$. Since $\delta>0$ is arbitrary, this proves the first claim.

Now we make the following key observation. For every bounded set $S \subset \GL_d(\R)$ and compact set $K \supset S$,  \begin{equation}\label{zareq} \beta(S)=\beta(Zar(S)\cap K)=\beta(\mathcal{H}(S)\cap K).\end{equation}
Here $\mathcal{H}(S)$ is the preimage in $\GL_d(\R)$ under $\rho$ of the linear span $\mathcal{H}_S$ of all $\rho(g), g \in S$, where $\rho$ is the linear representation with total space $E=\oplus_{k=1}^d \wedge^k \R^d$. This follows immediately from the following claim:\\

\noindent \underline{2nd claim:} There is $C=C(S)>0$ such that for all $\bw$ and $t$ we have:
\begin{equation}\label{inZar}\sup_{g \in \mathcal{H}(S) \cap K} \|a_tg\bw\| \leq C \sup_{g \in S}  \|a_tg \bw\|.\end{equation}
\noindent \emph{Proof of claim:} We note that $\mathcal{H}_S=\mathcal{H}_{Zar(S)}=\mathcal{H}_{Zar(S) \cap K}$, because $S$ and $Zar(S) \cap K$ have the same Zariski closure.  
Now consider the space $\mathcal{L}(\mathcal{H}_S, E)$  of linear maps from $\mathcal{H}_S$  to $E$. If $X \subset \mathcal{H}_S$ is a bounded set that spans $\mathcal{H}_S$, then the quantity $L \mapsto \sup_{ A \in X} \|L(A)\|$ defines a norm $|\cdot|_X$ on $\mathcal{L}(\mathcal{H}_S, E)$. Therefore for any two such sets $X,X'$ there is a constant $C_{X',X}>0$ such that for all $L \in  \mathcal{L}(\mathcal{H}_S, \wedge^k \R^d)$ we have $$|L|_{X'} \leq C_{X',X} |L|_{X}.$$
This applies in particular to the sets $X:=\rho(S)$ and $X':=\rho(Zar(S)\cap K).$ Now $(\ref{inZar})$ follows by setting $L:A \mapsto a_tA\bw$,  an element of  $\mathcal{L}(\mathcal{H}_S, E)$. This ends the proof of the second claim.

We may now finish the proof of Theorem \ref{rate}. Since $M$ is connected, it follows from the first claim that the $\mu$-almost sure value of $\mu_1(L)$ for $L \in M$ is unique and well-defined and equals $\beta(\phi(B))$ for any ball $B$ as in Theorem \ref{klw}. It is also equal to $\sup_{L \in M} \mu_1(L)$ by the first part of the claim. However, since $\phi$ is analytic, the Zariski closure of $\phi(B)$ coincides with $Zar(M)$. So $(\ref{zareq})$ entails $\beta(\phi(B))=\beta(\mathcal{H}(M) \cap K)$ for any compact  $K \supset \phi(B)$. But the first claim applied to $S=Zar(M) \cap K$ implies that $\mu_1(L) \leq  \beta(Zar(M) \cap K)$ for all $L \in Zar(M) \cap K$. Since $K$ is arbitrary, we get $\sup_{L \in Zar(M)} \mu_1(L) \leq \beta(Zar(M) \cap K)$. The right-hand side is the $\mu$-almost sure value of $\mu_1(L)$, so this inequality is an equality. This ends the proof.
\end{proof}

\bigskip

\subsection{Proof of Theorem \ref{parametrici}}\label{proof-para}

Without loss of generality we may assume that  $A_1 \ge \ldots \ge A_d$, where $A=\diag(A_1,\ldots,A_d)$ and $a_t=\exp(tA)$. Let $0=V_0<V_1<\ldots<V_h=\Q^d$ be the Harder-Narasimhan filtration associated to the submodular function $\tau_M$ on the grassmannian of $\Q^d$. Let $d_\ell=\dim V_\ell$ and $\Lambda_k=\frac{\tau_M(V_\ell)-\tau_M(V_{\ell-1})}{d_\ell-d_{\ell-1}}$ if $d_{\ell-1} <  k \leq d_\ell$. We need to show that for each $\ell=1,\ldots,h$ and for $\mu$-almost every $L$ the limit $(\ref{lim})$ holds  when $d_{\ell-1} <  k \leq d_\ell$ and that for $t$ large enough the first $d_{\ell-1}$ successive minima of $a_tL\Z^d$ are attained in $V_{\ell-1}$. Suppose this has been proved for all $\ell <i$ and let us prove it for $\ell=i$. 

By Minkowski's second theorem, for every $L \in M$
$$\limsup_{t \to +\infty} \frac{1}{t} \sum_{k \leq d_i} \log \lambda_k(a_tL\Z^d) \leq \tau_L(V_i) \leq \tau_M(V_i).$$
On the other hand we already know that  for $\mu$-almost every $L$
\begin{equation}\label{induu}\lim_{t \to +\infty} \frac{1}{t} \sum_{k \leq d_{i-1}} \log \lambda_k(a_tL\Z^d) = \sum_{k \leq d_{i-1}} \Lambda_k=\tau_M(V_{i-1}).\end{equation}
Hence 
$$\limsup_{t \to +\infty} \frac{1}{t} \sum_{d_{i-1} < k \leq d_i} \log \lambda_k(a_tL\Z^d) \leq \tau_M(V_i)- \tau_M(V_{i-1}).$$
Therefore to prove $(\ref{lim})$ it suffices to show that for $\mu$-almost every $L$
\begin{equation}\label{ts2}\liminf_{t \to +\infty} \frac{1}{t} \log \lambda_{d_{i-1}+1}(a_tL\Z^d) \geq \frac{\tau_M(V_i)- \tau_M(V_{i-1})}{d_i-d_{i-1}}=\Lambda_{d_i}.\end{equation}
This will also prove $(\ref{limm})$ for $\ell=i$ as we now explain. By Minkowski's theorem $$\lim_{t \to +\infty} \frac{1}{t}\sum_{k\leq d_{i-1}} \log \lambda_k(a_tLV_{i-1}(\Z)) = \tau_L(V_{i-1}) \leq \tau_M(V_{i-1}).$$ In view of $(\ref{induu})$ this quantity is actually equal to $\tau_M(V_{i-1})$. Therefore the $d_{i-1}$ first minima of $a_tL\Z^d$ are attained in $V_{i-1}$, and $(\ref{limm})$ follows from $(\ref{ts2})$. We now establish $(\ref{ts2})$ separating two cases.\\

\noindent\underline{First case:} $M=\{L\}$, with $L\in\GL_d(\QQ)$.

Consider the linear forms $L_1,\dots,L_d$ on $\R^d$ given by the rows of $L$.
They are linearly independent and have coefficients in $\QQ$.
Fix $\eps>0$ and consider the integer solutions $v \in \Z^d$ to the inequality
\begin{equation}\label{petitv}
\|a_tLv\|\leq e^{t(\Lambda_{d_i}-\eps)}.
\end{equation}
This implies
\begin{equation}\label{petitvv}
 \forall k\in\{1,\dots,d\},\quad e^{tA_{k}}|L_k(v)| \leq e^{t(\Lambda_{d_i}-\eps)}.
\end{equation}

Recall the subset of indices $I$ defined in $(\ref{taulimit2bis})$. For $k \notin I$  the restriction $L_k|_{V_{i-1}}$ of $L_k$ to  $V_{i-1}$ lies in the span of the $L_\ell|_{V_{i-1}}$ for $\ell < k$, $\ell \in I$. Since the linear forms $\{L_\ell|_{V_{i-1}}\}_{\ell \in I}$ are linearly independent, there is a unique choice of scalars $\alpha_{k,\ell} \in \overline{\Q} \cap \R$  such that $L_k|_{V_{i-1}} = \sum_{\ell \in I, \ell <k} \alpha_{k,\ell} L_\ell|_{V_{i-1}}$. We define $$M_k:=L_k - \sum_{\ell \in I,\ell<k} \alpha_{k,\ell} L_\ell.$$ By construction $M_k$ vanishes on $V_{i-1}$ and induces a linear form $\overline{M}_k$ on the quotient $\R^d/V_{i-1}$. Also by construction, the  linear forms $\{\overline{M}_k\}_{k \notin I}$ have coefficients in $\overline{\Q} \cap \R$ and are linearly independent on $\R^d/V_{i-1}$.

It follows from $(\ref{petitvv})$ that 

\begin{equation}\label{petitvvv}
 \forall k\notin I,\quad e^{tA_{k}}|\overline{M}_k(\overline{v})| \ll_L e^{t(\Lambda_{d_i}-\eps)}.
\end{equation}
where $\overline{v}=v\mod V_{i-1}$. But $\sum_{k \notin I} A_k = \tau_L(\R^d) - \tau_L(V_{i-1})$, so $$|\prod_{k\notin I} \overline{M}_k(\overline{v})| \ll_L e^{-(d-d_{i-1})t\eps},$$ because by definition of $V_i$ and $\Lambda_{d_i}$ we know that $(d-d_{i-1})\Lambda_{d_i}  \leq (\tau_L(\R^d) - \tau_L(V_{i-1}))$. On the other hand, $|\Lambda_{d_i}|\leq |A|:=\max |A_k|$, and in view of $(\ref{petitv})$ we have $\|\overline{v}\| \leq \|v\|\ll_L e^{2|A| t}$. So  setting $\eps'=\eps/(2d|A|)$,  we obtain
\begin{equation}\label{sch2} |\prod_{k\notin I} \overline{M}_k(\overline{v})| \ll_L \|v\|^{-\eps'}.\end{equation}
We are thus in a position to apply Schmidt's subspace theorem~\cite[Theorem~1F]{schmidt_da} and conclude that all integer solutions to this inequality, except finitely many, lie in a finite union of proper rational subspaces of $\R^d/V_{i-1}$  of minimal dimension.
Let $V/V_{i-1}$ be one of them. By definition of the Harder-Narasimhan filtration,   $\frac{\tau_L(V)-\tau_L(V_{i-1})}{\dim V - d_{i-1}}\geq\Lambda_{d_i}$. It follows that  solutions to $(\ref{petitvvv})$ satisfy\begin{equation}
 \forall k\notin I,\quad e^{tA_{k}}|\overline{M}_k(\overline{v})| \ll_L e^{t(\frac{\tau_L(V)-\tau_L(V_{i-1})}{\dim V - d_{i-1}}-\eps)}.
\end{equation}
And since $I \subset I_V$, solutions lying in $V/V_{i-1}$ satisfy \begin{equation}  |\prod_{k\in I_V\setminus I} \overline{M}_k(\overline{v})| \ll_L \|v\|^{-\eps'}.\end{equation} 
Since $V$ is rational and the linear forms $\{\overline{M}_k\}_{k \in I_V \setminus I}$ are linearly independent in restriction to $V/V_{i-1}$, we may apply Schmidt's subspace theorem once again and conclude that, apart from finitely many of them, the integer solutions are contained in a finite subset of proper rational subspaces. This would however contradict the minimality of $V$. In turn this implies that if $t$ is large enough all integer solutions to $(\ref{petitv})$ are contained in $V_{i-1}$.  This shows $(\ref{ts2})$  as desired.\\

\noindent\underline{General case:}\\
Note that $Zar(M)$ is an irreducible algebraic variety, so $\tau_M$ is submodular. In view of $(\ref{taulimit2})$, given any rational subspace $V$ in $\R^d$ the set of elements $L$ in $M$ such that $\tau_L(V) \neq \tau_M(V)$ is a proper closed subvariety of $Zar(M)$ that is defined over $\Q$ and of bounded degree. By Lemma \ref{countable} below (applied to $F=\R \cap \overline{\Q}$) we may choose a point $L_0$ of $Zar(M)$ in $\GL_d(\QQ)$ such that $\tau_{L_0}(V)=\tau_M(V)$ for all rational $V$. Then the Harder-Narasimhan filtrations and the Grayson polygons of $M$ and $\{L_0\}$ coincide. 
By the first part of the proof,
\begin{equation}\label{firstpart} \liminf_{t\to+\infty} \frac{1}{t}\log \lambda_{d_{i-1}+1}(a_tL_0\Z^d) \geq \Lambda_{d_i}.\end{equation}
Now we may invoke Theorem \ref{rate}. In view of $(\ref{firstpart})$ and $(\ref{induu})$ this implies $(\ref{ts2})$.

\begin{lemma}[countable unions of proper subvarieties]\label{countable} Let $X$ be an irreducible algebraic variety defined over a number field $K$. Let $F$ be an algebraic extension of $K$ of infinite degree.  Let $k \ge 1$ and suppose $(X_j)_{j \ge 1}$ is a countable family of proper closed subvarieties of $X$ of degree at most $k$, each defined over a field $K_j$ of degree at most $k$ over $K$.  Then $X(F)$ is not contained in the union of the $X_j(\overline{\Q})$, $j \ge 1$. 
\end{lemma}

\begin{proof} Looking at a finite cover by affine varieties, we may assume that $X$ is affine. Then by Noether's normalization theorem, there is a finite morphism $f: X \to \mathbb{A}^d$ defined over $\QQ$, where $d=\dim X$. So again without loss of generality, we may assume that $X=\mathbb{A}^d$, the $d$-dimensional affine space. But we can of course find elements $x_i\in F$, $i=1,\ldots,d$, such that $x_i$ has degree $>k^2$ over $K(x_1,\ldots,x_{i-1})$. Then $K_j(x_1,\ldots,x_i)$ has degree $>k$ over $K_j(x_1,\ldots,x_{i-1})$ for all $j$, so $(x_1,\ldots,x_d)$ will not belong to any $X_j$. 
\end{proof}

\subsection{The subspace theorem}\label{proof-sub}

We are now ready for the proof of Theorem \ref{subspace-main}.  

\begin{proof}[Proof of Theorem \ref{subspace-main}]  Without loss of generality, we may assume that $M$ is a bounded subset of $\GL_d(\R)$. In particular there is $C=C(M)\ge 1$ such that $\|L\bx\| \leq C \|\bx\|$ for every $\bx \in \R^d$ and $L \in M$. We are going to show that there is a finite set $P$, with $|P| \leq (2Cd^2\eps^{-1})^d$ of one-parameter unimodular diagonal semigroups $a=(a_t)_{t \ge 0}$ with the following property. If $L \in M$ and $\bx \in \Z^d$ is a solution to $(\ref{total})$ such that the integer part $t$ of $\frac{\eps}{4d}\log \|\bx\|$ is at least $1$, then there is $a \in P$ such that
\begin{equation}\label{aas} \|a_tL\bx\| \leq de^{- t}.\end{equation}

To see this let $\ell_i=\log |L_i(\bx)|$. By $(\ref{total})$ we have $\ell_1 + \ldots + \ell_d \leq -4d t$. Note that $\ell_i \leq \log\|x\| + \log C\leq (\frac{8d}{\eps}+\log C)t$. Let $\ell_i' = \max \{\ell_i, -Dt\}$, where $D=5d(\log C+\frac{8d}{\eps})$. If there is an index $i$ such that $\ell_i \neq \ell_i'$, then $\ell_1'+\ldots+ \ell_d' \leq -D t + (d-1)(D/5d)t \leq -4d t$. We conclude that $\ell_1'+\ldots+ \ell_d' \leq -4d t$ always.  Now define $b_i:= \frac{1}{d}(\ell_1'+ \ldots + \ell_d') - \ell_i'$. Then $b_1+\ldots+b_d=0$, and for each $i$ \[\ell_i + b_i \leq \ell_i'+b_i \leq  - 4t.\]
 On the other hand $|\ell_i'| \leq Dt$, so $|b_i| \leq 2Dt$. Let $B$ be set of integers points in $\Z^d$ with coordinates in $[-3D,3D]$. Choose $(n_1,\ldots,n_d) \in B$ such that $|\frac{b_i}{t}-n_i| \leq \frac{1}{2}$ for all $i$. In particular $|\sum_1^d n_i| \leq d/2$. Changing some $n_i$ to the next or previous integer if needed, we may ensure that $\sum_1^d n_i=0$ and $|\frac{b_i}{t}-n_i| \leq \frac{3}{2}$. Then we set $a_t=\diag(e^{n_1t},\ldots,e^{n_dt})$ and let $P$ be the finite set of such diagonal semigroups. Note that $|P|\leq (6D)^d$. Clearly $\ell_i + n_it \leq -t$ and $(\ref{aas})$ follows. 

Now we may apply Corollary~\ref{paraweak} and conclude  that  for $\mu$-almost every $L \in M$, if $\bx$ is a large enough solution of $(\ref{total})$, it must lie inside $V(a)$ for some $a$ in $P$. This shows that the number of exceptional rational subspaces is finite. However by Lemma \ref{values} above  the subspace $V(a)$ can take at most $b(2^d)$ possible values as $a$ varies among all unimodular diagonal semigroups. This ends the proof.  
\end{proof}

\begin{remark} Note that conversely each $V(a)$, and hence each $V_i$ in Theorem \ref{subspace-main}, contains infinitely many solutions to $(\ref{total})$ for every $\eps>0$.
\end{remark}

\begin{remark}
Furthermore the rational subspaces $V_1\cup\dots\cup V_r$ depend only on the rational Zariski closure of $M$ and not on the choice of $L$. And because they are defined by a simple slope condition their height is effectively bounded in terms of the height of $Zar(M)$.  On the other hand, the finite set of exceptional solutions lying outside  the $V_i$  depends on $L$ and $\eps$ and there is no known bound on their height or number, see \cite[Prop. 5.1]{evertse-survey}. When $M$ is a single point it is however possible to group together the finitely many exceptional solutions into another set of proper subspaces  whose number, but not height, can be effectively bounded, see \cite{evertse}.
\end{remark}


\subsection{Varieties defined over $\R$}\label{R-sec}
What happens if we remove the assumption that the Zariski  closure of $M$ is defined over $\QQ$ in  Theorem \ref{parametrici} ?  Without this assumption, diagonal flow trajectories may not behave as nicely and typically no limit shape is to be expected. However we may give a simple upper and lower bound on the almost sure value of $\mu_k(L)$ for $k=1,\ldots,d$, which exists by Theorem \ref{rate}, in terms of the rational and real Grayson polygons as we now discuss. So far we have only considered the rational Harder-Narasimhan filtration $\{V^{\Q}_i\}_{i=0}^{h_\Q}$ and its rational polygon $\mathcal{G}^{\Q}$ with slopes $s_i^{\Q}$, because we have restricted ourselves to considering the grassmannian of rational subspaces. But we may also take $k=\R$ in \S \ref{h-n-sec} with the same submodular function $\tau_M$. This yields a new Harder-Narasimham filtration $\{V^{\R}_i\}_{i=0}^{h_\R}$ for the real field and a new Grayson polygon $\mathcal{G}^{\R}$ with slopes $s_i^{\R}$, that obviously lies \emph{below} the rational polygon.  

Let as before $M=\phi(U)$ be the image of a connected open set $U \subset \R^n$ under an analytic map $\phi:U \to \GL_d(\R)$ and $\mu$ the measure on $M$ that is the image of the Lebesgue measure on $U$. In this paragraph we no longer assume that $Zar(M)$ is defined over $\QQ$. Let $\mu_k$ be the $\mu$-almost sure value of $\mu_k(L)$ as given by Theorem \ref{rate} and $\mu_k^{sup}$ the supremum over all $L \in M$ of  $\mu_k^{sup}(L)$, where  $\mu_k^{sup}(L)$ is defined by the same formula as $\mu_k(L)$ with a limsup in place of the liminf. Consider the points $(k,\mu_k)$ for $k=1,\ldots,d$ and interpolate linearly between them, so as to form a polygon $\mathcal{G}_\mu$ as in Fig \ref{hn}. Similarly form $\mathcal{G}_\mu^{sup}$ with $(k,\mu^{sup}_k)$. Note that $\mathcal{G}_\mu^{sup}$ is convex (being a supremum of convex polygons), but $\mathcal{G}_\mu$ may not be.

\begin{proposition}[``Sandwich theorem'']\label{pol-comp} The polygons $\mathcal{G}_\mu,\mathcal{G}^{sup}_\mu$ lie \emph{in between} the rational Grayson polygon $\mathcal{G}^{\Q}$ and the real Grayson polygon $\mathcal{G}^{\R}$. In other words, for each $k=1,\ldots d$, $$\sum_{i\leq k} s_{m_i}^{\R} \leq \mu_k \leq  \mu_k^{sup} \leq \sum_{i\leq k} s_{n_i}^{\Q},$$ where $m_i,n_i$ defined by  $\dim V^{\R}_{m_i-1} \leq i <\dim  V^{\R}_{m_i}$ and  $\dim V^{\Q}_{n_i-1} \leq i < \dim V^{\Q}_{n_i}$. 
\end{proposition}

\begin{figure}[h]
\begin{center}
\begin{tikzpicture}
\draw[->] (-1,0) -- (5,0);
\draw (0,-3) -- (0,0.2);
\draw[->] (0,0.37) -- (0,1);
\foreach \x in {0,...,4}
{
\draw (\x,-0.1) -- (\x,0.1) node[anchor=south] {\tiny{\x}};
}
\draw[red,thin] (0,0) -- (1,-1.2) -- (2,-1.7) -- (3,-2) -- (4,-.5) -- (5,1);
\filldraw[red] (0,0) circle (1pt);
\filldraw[red] (1,-1.2) circle (1pt);
\filldraw[red] (1.2,-0.9) node {$V_1^\Q$};
\filldraw[red] (2,-1.7) circle (1pt);
\filldraw[red] (2,-1.3)  node {$V_2^\Q$};
\filldraw[red] (3,-2) circle (1pt) ;
\filldraw[red] (2.9,-1.6)  node {$V_3^\Q$};
\filldraw[red] (5,1) circle (1pt);
\filldraw[red] (5.9,1)  node {$V_4^\Q=V_5^\R$};
\filldraw[red] (4,-.5) circle (1pt);

\draw[blue,thin] (0,0) -- (1,-2) -- (2,-2.7) -- (3,-3) -- (4,-2) -- (5,1);
\filldraw[blue] (0,0) circle (1pt);
\filldraw[blue] (1,-2) circle (1pt);
\filldraw[blue] (1,-2.5) node {$V_1^\R$};
\filldraw[blue] (2,-2.7) circle (1pt);
\filldraw[blue] (2,-3)  node {$V_2^\R$};
\filldraw[blue] (3,-3) circle (1pt) ;
\filldraw[blue] (3.2,-3.2)  node {$V_3^\R$};
\filldraw[blue] (4,-2) circle (1pt);
\filldraw[blue] (4.2,-2.3)  node {$V_4^\R$};

\draw[black,thin] (0,0) -- (1,-1.5) --  (3,-2.5) --  (5,1);
\filldraw[black] (1,-1.5) circle (1pt);
\filldraw[black] (3,-2.5) circle (1pt) ;
\filldraw[black] (2,-2) circle (1pt) ;
\filldraw[black] (4,-.75) circle (1pt);

\end{tikzpicture}
\caption{$\mathcal{G}_\mu$ lies between the rational and the real polygons}
\end{center}
\end{figure}
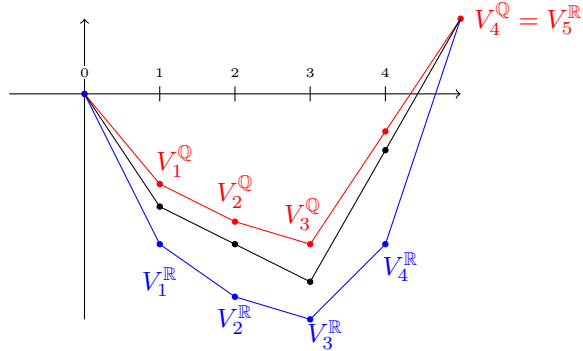

\begin{proof}The upper bound follows from Minkowski's theorem: the first $d_i$ successive minima in $a_tL\Z^d$ are smaller than those attained in $V^{\Q}_i$, so
\[
\mu_k(L) \leq \limsup_{t \to +\infty} \frac{1}{t} \sum_{k \leq d_i} \log \lambda_k(a_tL\Z^d) \leq \tau_L(V_i^{\Q}) \leq \tau_M(V_i^{\Q}).
\]
Since this holds for each $i=0,\ldots,h_{\Q}$, and the polygons are convex, we get that $\mathcal{G}_\mu^{sup}$ lies below $\mathcal{G}^{\Q}$.

The lower bound follows from a modified version of the proof of Theorem \ref{rate} that uses instead the refined quantitative non-divergence estimates for successive minima already mentioned and established in \cite[Chapter 6]{saxce-hdr} or  \cite[Theorem 5.3]{lmms}.  We now explain this briefly and refer to \cite[Theorem 7.3.1]{saxce-hdr} for the details of a more general statement.   Let $B$ be a small ball around some $x \in U$. Note that $\phi(B)$ is Zariski dense in $M$ and thus $\tau_M=\tau_{\phi(B)}$. Let $\beta_k$ be the infimum of all $\tau_M(W)$, where $W$ ranges among real subspaces with dimension $k$. In view of $(\ref{taulimit4})$ this implies that there is $I$ with $I(a) \ge \beta_k$ such that $(L\bw)_I \neq 0$ for some $L \in M$. By compactness of the grassmannian of $k$-dimensional subspaces there is $c>0$ such that $\max_{I, I(a) \ge \beta_k}\sup_{L \in M} \|(L\bw)_I\| \ge c \|\bw\|$ for all $W$. It follows in particular that $\sup_{L\in M} \|a_tL\bw\| \gg e^{t\beta_k}$ uniformly in $t>0$ and in non-zero integer valued $\bw$. Now let $\gamma_k \leq \beta_k$ be the largest such function of $k$ that is convex in $k$. Then the exact same Borel-Cantelli argument used in the proof of Claim 1 in Theorem \ref{rate}, using instead the refined quantitative non-divergence in the form of \cite[Theorem 5.3]{lmms} shows  that for $\mu$-almost every $L$ in $B$, $\mu_k(L) \ge \gamma_k$.  Since $\beta_{d_i}=\tau_M(V_i^{\R})$, we have shown that $\mathcal{G}_\mu$ lies above $\mathcal{G}^\R$.
\end{proof}

This proposition also holds for measures $\mu$ on $\GL_d(\R)$ which are good in the sense of \S \ref{mgm} below, see \cite[Theorem 7.3.1]{saxce-hdr}.

If $Zar(M)$ is defined over $\Q$, then by uniqueness of the Harder-Narasimhan filtration, we see that the real and rational filtrations coincide. In particular, in this case all three polygons coincide. If $Zar(M)$ is defined over $\QQ$, then the filtration over $\R$ is in fact defined over $\QQ$, and Theorem \ref{parametrici} asserts that $\mathcal{G}_\mu$ coincides with $\mathcal{G}^\Q$. However $\mathcal{G}^\R$ may be different. 

Similarly:

\begin{corollary} Let $M$ be as in Theorem \ref{subspace-main}, except we no longer assume that $Zar(M)$ is defined over $\QQ$. Let $a=(a_t=\exp(tA))_t$ be a diagonal flow. Assume that $\R^d$ is $M$-semistable with respect to $a$ (see Def. \ref{semi}). Then for $\mu$-almost every $L \in M$,
$$\liminf_{t \to +\infty} \lambda_1(a_tL\Z^d)\ge \frac{1}{d}\trace(A).$$
\end{corollary}

\subsection{Pl\"ucker closure}\label{schubert-sec}

In this paragraph we define the notion of Pl\"ucker closure of a subset $M \subset \GL_d(\R)$ and we explain why in Theorems \ref{subspace-main} and \ref{parametrici} instead of assuming that $Zar(M)$ is defined over $\QQ$, it is enough to assume that the \emph{Pl\"ucker closure} of $M$ is defined over $\QQ$. 

Let $M \subset \GL_d(\R)$ be a subset and $(E,\rho)$ be the direct sum of the exterior power representations of $\GL_d$, namely $E=\oplus_{k=1}^d \wedge^k \R^d$. We denote by $\mathcal{H}_M$  the $\R$-linear span in $\End(E)$ of all $\rho(g), g \in M$. We further define the \emph{Pl\"ucker closure} $\mathcal{H}(M)$ of $M$ as the inverse image of $\mathcal{H}_M$ under $\rho$ in $\GL_d(\R)$. Note that $\mathcal{H}(M)$ contains the Zariski closure $Zar(M)$ of $M$.  We say that $\mathcal{H}(M)$ is defined over $\QQ$ if $\mathcal{H}_N$ has a basis with coefficients in $\QQ$ in the canonical basis of $E$.  An obvious sufficient condition for this to hold is to ask for $Zar(M)$ to be defined over $\QQ$.

We also say that $M$ is \emph{Pl\"ucker irreducible} if $\rho(M)$ is not contained in a finite union of proper subspaces of $\mathcal{H}_M$ in $\End(E)$. Clearly Zariski irreducibility implies Pl\"ucker irreducibility. It is clear that Pl\"ucker irreducibility of $M$ is enough to guarantee that $\tau_M$ is submodular by the argument of Lemma \ref{tausm}. It is also clear that $\tau_M=\tau_{\mathcal{H}(M)}$. 

For simplicity of exposition in this paper, we have chosen to state the assumptions in our main theorems in terms of the Zariski closure of $M$, but in fact Theorems \ref{subspace-main} and \ref{parametrici} hold assuming only that the Pl\"ucker closure is defined over $\QQ$. The proof is verbatim the same as the one we have given, except that in Theorem \ref{rate}, we get the following slightly stronger statement: for $\mu$-almost every $L$ in $M$ and each $k$,
$$\mu_k(L)=\sup_{L' \in \mathcal{H}(M)} \mu_k(L').$$
Again the proof of this equality is exactly the one given for Theorem \ref{rate} when $k=1$. However the reduction to the case $k=1$ via Mahler's lemma no longer works here, because $\mathcal{H}(M)$ may differ from $\mathcal{H}(\rho(M))$. Instead one may use the enhanced version of the quantitative non-divergence estimates already mentioned in the proof of Proposition \ref{pol-comp}, namely \cite[Chapter 6]{saxce-hdr} or \cite[Theorem 5.3]{lmms} in place of Theorem \ref{klw}, which enables one to run the argument simultaneously for all $k$.  This shows that the almost sure value of each $\mu_k(L)$, and thus the polygon $\mathcal{G}_\mu$ defined in the previous paragraph, depend only on the Pl\"ucker closure of $M$. 

There is another notion of envelope of $M$ that is also natural to consider, namely the intersection $Sch(M)$ of all translates $\mathcal{S}_\sigma g$ of Schubert varieties $\mathcal{S}_\sigma=\overline{B\sigma B}$  containing $M$. Here $B$ is one of the $d!$ Borel subgroups containing the diagonal subgroup and the closure is the Zariski closure. This \emph{Schubert closure} contains the Pl\"ucker closure. It is easy to see from $(\ref{taulimit2bis})$  that the submodular function $\tau_M$ depends only on $Sch(M)$. Restricting to rational translates $\mathcal{S}_\sigma g$ with $g \in \GL_d(\Q)$ one obtains the \emph{rational Schubert closure} $Sch^{\Q}(M)$. The asymptotic shape and the exceptional subspaces appearing in Theorems \ref{subspace-main} and \ref{parametrici} depend only on $Sch^{\Q}(M)$. A natural question we could not answer is whether or not the main theorem remains valid under the (weaker) assumption that $Sch(M)$ is defined over $\QQ$. The answer is clearly yes when $Sch(M)$ is defined over $\Q$ as follows readily from Proposition \ref{pol-comp}.

\subsection{More general measures}\label{mgm}

In this paragraph, we define a class of measures called \emph{good measures}, which is wider than the family considered so far of push-forwards of the Lebesgue measure under analytic maps, and for which Theorems \ref{subspace-main}, \ref{parametrici} and \ref{rate} continue to hold. This class is very closely related to the so-called friendly measures of \cite{kleinbock-lindenstrauss-weiss, kleinbock-tomanov}: instead of being expressed in terms of an affine span, the non-degeneracy condition is defined in terms of local Pl\"ucker closures.

First we need to recall some piece of terminology. We fix a metric on $\GL_d(\R)$, say induced from the euclidean metric on the matrices $M_d(\R)$.  Given two positive parameters $C$ and $\alpha$, a real-valued function $f$ on the support of $\mu$ is called \emph{$(C,\alpha)$-good with respect to $\mu$} if for any ball $B$ in $\GL_d(\R)$ and all $\eps>0$
\[ \mu(\{x\in B, |f(x)|\leq\eps\|f\|_{\mu,B}\}) \leq C\eps^\alpha \mu(B),\]
where $\|f\|_{\mu,B}=\sup_{x\in B\cap\Supp\mu} |f(x)|$.
The measure $\mu$ is \emph{doubling} on a subset $X\subset \GL_d(\R)$ if there exists a constant $C'$ such that for every ball $B(x,r)\subset X$,
\[ \mu(B(x,2r)) \leq C'\mu(B(x,r)).\]
Then we say that a Borel measure $\mu$ is \emph{locally good} at $L \in \GL_d(\R)$ if there exists a ball $B$ around $L$ and positive constants $C,\alpha$ such that 
\begin{enumerate}[(i)]
\item The measure $\mu$ is doubling on $B$.
\item For every $k\in\{1,\dots,d\}$, for every pure $k$-vector $\bw=\bv_1\wedge\dots\wedge\bv_k$ in $\wedge^k\R^d$, and every $a\in \GL_d(\R)$, the map $y\mapsto \|ay\cdot\bw\|$ is $(C,\alpha)$-good on $B$ with respect to $\mu$.
\end{enumerate}

Recall next that given a subset $S\subset \GL_d(\R)$ and a point $x\in S$, the \emph{local Pl\"ucker closure} $\mathcal{H}_x(S)$ of $S$ at $x$ is the intersection over $r>0$ of the Pl\"ucker closures of $S\cap B(x,r)$:
\[ \mathcal{H}_x(S) = \bigcap_{r>0} \mathcal{H}(S\cap B(x,r)).\]

We may now define the class of good measures.

\begin{definition} A locally finite Borel measure $\mu$ on $\GL_d(\R)$ will be called a \emph{good measure} if $\Supp \mu$ is Pl\"ucker irreducible and if it satisfies the following assumptions for $\mu$-almost every $L$:
\begin{enumerate}
\item The measure $\mu$ is locally good at $L$;
\item The local Pl\"ucker closure of $\Supp\mu$ at $L$ is equal to that of $\Supp \mu$.
\end{enumerate}
\end{definition}

Of course, the most important example of a good measure is that of the push-forward of the Lebesgue measure under an analytic map.

\begin{proposition}[Analytic measures are good]\label{analyticU}
Let $n\in\N$, let $U$ be a connected open set in $\R^n$, and let $\varphi:U\to\GL_d(\R)$ be a real-analytic map.
Then the push-forward under $\varphi$ of the Lebesgue measure on $U$ is a good measure on the Zariski closure $M$ of $\varphi(U)$.
\end{proposition}

\begin{proof} Note that the maps of the form $u \mapsto \|a\varphi(u)\cdot \bw\|^2$ are linear combinations of products of matrix coefficients of $\varphi(u)$ and hence belong to a finite dimensional linear subspace of analytic functions on $U$ independent of the choice of $a \in \GL_d(\R)$ and $\bw \in \wedge^* \R^d$. So  \cite[Proposition~2.1]{kleinbock_dichotomy} applies.
\end{proof}

Theorem \ref{rate} holds for all good measures $\mu$ on $\GL_d(\R)$ with the same proof (suitably modified via the enhanced quantitative non-divergence estimates as mentioned in the previous paragraph). In fact the definition of good measures has been tailored precisely for Theorem \ref{rate} to hold. Thus the subspace theorem for manifolds and  the parametric subspace theorem, Theorems \ref{subspace-main} and \ref{parametrici}, hold for good measures $\mu$ such that the Pl\"ucker closure of $\Supp(\mu)$ is defined over $\QQ$.


\section{A generalization to number fields}
\label{section:nf}

Schmidt himself observed in \cite{schmidt_nf} that his theorem for the field $\Q$ of rationals implied a more general version for any number field $K$, and this was generalized shortly after by Schlickewei \cite{schlickewei_padic}, who gave a statement allowing also finite places.  In this section, we formulate a similar generalization of Theorems~\ref{subspace-main} and \ref{parametrici}.

\bigskip

For a place $v$ of a number field $K$, the completion of $K$ at $v$ is denoted by $K_v$.
As in Bombieri-Gubler \cite[\S~1.3-1.4]{bombierigubler} we use the following normalization for the absolute value $|\cdot|_v$ on $K_v$
\[ |x|_v := N_{K_v/\Q_v}(x)^{\frac{1}{[K:\Q]}},\]
where $\Q_v$ is the completion of $\Q$ at the place $v$ restricted to $\Q$ and $N_{K_v/\Q_v}(x)$ is the norm of $x$ in the extension $K_v/\Q_v$. The product formula then reads $\prod_v |x|_v=1$ for all $x \in K$.  If $S$ is a finite set of places of $K$ containing all archimedean ones, the ring $\mathcal{O}_{K,S} \subset K$ of $S$-integers is the set of $x \in K$ such that $|x|_v \leq 1$ for all places $v$ lying outside  $S$. Elements of its group of units are called $S$-units.

Let $d$ be a positive integer.
The $d$-dimensional space $K_v^d$ will be endowed with the supremum norm $\|\cdot\|_v$, i.e. for an element $\bx=(x_1^{(v)},\dots,x_d^{(v)})$ in $K_v^d$, $\|\bx\|_v=\max_{1\leq i\leq d} |x_i^{(v)}|_v$.
More generally, we let $K_S=\prod_{v\in S} K_v$ be the product of all completions of $K$ at the places of $S$, and if $\bx=(\bx^{(v)})_{v\in S}$ is an element of $K_S^d$, we define its \emph{norm} $\|\bx\|$, its \emph{content} $c(\bx)$ and \emph{height} $H(\bx)$ by
\[ \|\bx\| = \max_{v\in S} \|\bx^{(v)}\|_v
\mbox{,}\quad c(\bx)=\prod_{v\in S} \|\bx^{(v)}\|_v\quad\mbox{and}\quad  H(\bx)=\prod_{v\in S} \max\{1,\|\bx^{(v)}\|_v\}.\]

It is clear that  $c(\bx) \leq H(\bx)$ and $\|\bx\| \leq H(\bx) \leq \max\{1,\|\bx\|^{|S|}\}$. It follows from the product formula that $\|\bx\| \ge 1$ if $0 \neq \bx \in \mathcal{O}_{K,S}$.  The image of $\cO_{K,S}$ in $K_S$ under the diagonal embedding is discrete and cocompact in $K_S$ \cite[Chapter VII]{lang_ant}, and closed balls in $K_S$ for the norm are compact.  Furthermore, it is easily seen from Dirichlet's unit theorem that there is a constant $C=C(K,S)>0$ such that for all $\bx \in K_S^d$ with $c(\bx) \neq 0$, there is an $S$-unit $\alpha \in \mathcal{O}_{K,S}$ such that
\begin{equation}\label{diri}
e^{-C} \|\alpha \bx\|^{|S|} \leq c(\bx) \leq \|\alpha \bx\|^{|S|}.
\end{equation} 
For each $v\in S$, we denote by $\GL_d(K_v)$ the group of invertible $d\times d$ matrices with coefficients in $K_v$, and we set
\[ \GL_d(K_S) = \prod_{v\in S} \GL_d(K_v).\]

A product measure $\mu=\otimes_{v \in S} \mu_v$ on $\GL_d(K_S)$ will be called a \emph{good measure} if each $\mu_v$ is a good measure on $\GL_d(K_v)$. The definition of a good measure given in \S \ref{mgm} for $K_v=\R$ extends verbatim to other local fields $K_v$. 

Examples of good measures are provided by push-forwards of Haar measure under strictly analytic maps. Indeed Proposition \ref{analyticU} continues to hold for analytic maps $\phi:U \to \GL_d(K_v)$ whose coordinates are defined by convergent power series on a ball $U:=B(x,r)=\{y \in K_v^n, \max_{i=1}^n|y_i-x_i|\leq r\}$ in $K_v^n$. This is because, on the one hand the push-forward of Haar measure on $K_v$ under $\phi$ will be locally good everywhere  by \cite[Prop. 4.2]{kleinbock-tomanov}, and on the other hand, the Pl\"ucker closure of the image of a ball $B \subset U$ of positive radius is independent of the ball, because convergent power series that vanish on an open ball must vanish everywhere. 

Let  $E_v:=\oplus_{k\leq d} \wedge^k K_v^d$. We will say that the Pl\"ucker closure of $\mu$ is defined over $\QQ$ if for each $v \in S$, the subspace $\mathcal{H}_{\Supp \mu_v}$ of $\End(E_v)$ (see \S \ref{schubert-sec}) is defined over $\QQ$, i.e. is the zero set of a family of linear forms on $\End(E_v)$ with coefficients in $\QQ \cap K_v$. Clearly this is the case if for each $v \in S$, the Zariski closure of $\Supp \mu_v$ in $\GL_d(K_v)$ is defined over $\QQ$.  We will also denote by $\mathcal{H}(\mu)$ the cartesian product of all $\mathcal{H}(\mu_v)$, where $\mathcal{H}(\mu_v):=\{g \in \GL_d(K_v), \rho_v(g) \in \mathcal{H}_{\Supp \mu_v}\}$.

We are now ready to state:

\begin{theorem}[Subspace theorem for manifolds, $S$-arithmetic version]
\label{subspacenf}
Let $K$ be a number field, $S$ a finite set of places including all archimedean ones, $\mathcal{O}_{K,S}$ its ring of $S$-integers, and $d$ in $\N$.
Let $\mu$ be a good measure on $\GL_d(K_S)$ whose Pl\"ucker closure is defined over $\QQ$.
Then there are proper subspaces $V_1,\dots,V_r$ of $K^d$ such that for $\mu$-almost every $L$ and for every $\eps>0$, 
the inequality
\begin{equation}\label{totalK} \prod_{v\in S}\prod_{i=1}^d |L_i^{(v)}(\bx)|_v \leq \frac{1}{c(\bx)^{\eps}}\end{equation}
has only finitely many solutions $\bx \in \mathcal{O}_{K,S}^d\setminus(V_1\cup\dots\cup V_r)$ up to multiplication by an $S$-unit.
\end{theorem}
Here $L=(L^{(v)})_{v \in S} \in \GL_d(K_S)$ and each $L_i^{(v)}$ denotes the $i$-th row of the matrix $L^{(v)} \in \GL_d(K_v)$. Note that the left and right-hand sides of $(\ref{totalK})$ are unchanged if $\bx$ is changed into $\alpha \bx$ for some $S$-unit $\alpha$, so we will focus on the equivalence classes of solutions. The bound on the number $r$ of exceptional subspaces depends only on $d$ and $|S|$, and each subspace contains infinitely many (classes of) solutions to $(\ref{totalK})$.

When $\mu$ is a Dirac mass at a point $L=(L^{(v)})_{v \in S}$ with $L^{(v)} \in \GL_d(\QQ) \cap \GL_d(K_v)$, then the theorem is exactly the $S$-arithmetic Schmidt subspace theorem as stated in \cite[7.2.5]{bombierigubler}.



\subsection{Parametric version}

Theorem \ref{subspacenf} is deduced from Theorem \ref{parametricnf} below, which is a parametric version analogous to Theorem \ref{parametrici}. To formulate it, we need to define the $S$-arithmetic analogues of a lattice and its successive minima. A family of vectors $\bx_1,\dots,\bx_k$ in $K_S^d$ is said to be \emph{linearly independent} if it spans a free $K_S$-submodule of rank $k$. Equivalently the vectors $\bx_1^{(v)},\dots,\bx_k^{(v)}$ are linearly independent over $K_v$ for each place $v$.

\begin{definition}[Lattice in a number field]\label{subla}
For any positive integers $k\leq d$, we define a \emph{sublattice in $K_S^d$} of rank $k$ to be a discrete free
$\cO_{K,S}$-submodule of rank $k$ in $K_S^d$.
In other words, it is a subgroup $\Delta \leq K_S^d$ that can be written
\[ \Delta = \cO_{K,S}\bx_1\oplus\dots\oplus \cO_{K,S}\bx_k\]
for some linearly independent elements $\bx_i=(\bx_i^{(v)})_{v\in S}$ in $K_S^d$. If $k=d$ we say that $\Delta$ is a lattice in $K_S^d$.
\end{definition}

If $L$ is any element of $\GL_d(K_S)$, then $\Delta=L\cO_{K,S}^k$ is a sublattice of rank $k$ in $K_S^d$. Conversely they are all of this form.
%
%
%
Note that if $\bx_1,\dots,\bx_k$ are vectors from a lattice $\Delta$, they are linearly independent if and only if $\bx_1^{(v)},\dots,\bx_k^{(v)}$ are linearly independent for \emph{some} place $v$.

\begin{definition}[Successive minima] If $\Delta$ is a sublattice in $K_S^d$ and $k\in[1,d]$,
\[ \lambda_k(\Delta)=
\inf \{\lambda>0 \,|\, \exists \bx_1,\dots,\bx_k\in\Delta\ 
\mbox{linearly independent with}\
\forall j,\  c(\bx_j)\leq\lambda\} \] is its \emph{$k$-th successive minimum},
where $c(\bx_j)$ is the content defined earlier. \end{definition} It follows from $(\ref{diri})$ that we may have defined the successive minima using $\|\bx_j\|^{|S|}$ in place of $c(\bx_j)$ without much difference. In particular either definition will suit the theorem below.

The analogue of Theorem \ref{parametrici} now reads as follows.  Fix a diagonal element $a=\diag (a_i^{(v)})$ in $\GL_d(K_S)$, and consider the flow $a_t=a^t$, for every $t\in\N$.

\begin{theorem}[$S$-arithmetic strong parametric subspace theorem]
\label{parametricnf} Let $K$ be a number field and $S$ a finite set of places containing all archimedean ones. 
Let $(a_t)_{t \in\N}$ be a diagonal flow in $\GL_d(K_S)$ and $\mu$ a good measure on $\GL_d(K_S)$ whose Pl\"ucker closure is defined over $\QQ$. Then there are $K$-subspaces $0=V_0<V_1<\dots<V_h=K^d$ and real numbers $s_1 < \ldots <s_h$ such that for each $i=1,\ldots,h$ and for $\mu$-almost every $L$,
\begin{enumerate}
\item If $\dim V_{i-1}< k\leq \dim V_i$, then $\lim_{t \to +\infty}\frac{1}{t}\log\lambda_k(a_tL(\cO_{K,S})^d)=s_i$.
\item For all $t>0$ large enough, the first $\dim V_i$ successive minima of $a_tL(\cO_{K,S})^d$ are attained in $V_i$.
\end{enumerate}
In other words for every $\eps>0$ and $\mu$-almost every $L$, there is $t_{\eps,L}$ such that if $t>t_{\eps,L}$, $\ell=1,\ldots,h,$ and $\bx\in(\mathcal{O}_{K,s})^d$,
\begin{equation} c(a_tL\bx) \leq e^{t(s_\ell-\eps)}
\quad\Rightarrow\quad
\bx\in V_{\ell-1}.\end{equation}
\end{theorem}

The $K$-subspaces $V_i$, $i=1,\ldots,h$, appearing in Theorem \ref{parametricnf} are the terms of the Harder-Narasimhan filtration associated to $M:=\Supp \mu \subset \GL_d(K_S)$ and the quantities $s_i$ are its slopes. We describe this filtration in the next paragraph.

\subsection{Harder-Narasimhan filtrations for $K_S^d$}\label{HS-K}

In this paragraph we associate a submodular function on the grassmannian $\Grass(K^d)$ that generalizes the expansion rate  $\tau_M$ defined earlier in \S \ref{exp-rate-sec}. Given a $K_S$-submodule $V$ in $K_S^d$ we define its expansion rate as follows 
$$\tau(V):= \lim_{t \to +\infty} \frac{1}{t}\log c(a_t \bv)$$
where $\bv \in \wedge^* K_S^d = \prod_{v \in S} \wedge^* K_v $ and $c(\bv)$ is the content as defined earlier. Here we identify $\wedge^k K_v$ with $K_v^N$, $N={{d\choose k}}$ and use the standard basis $e_{i_1} \wedge \ldots \wedge e_{i_k}$ to define the norm, and note that $V=\prod_v V_v$, where each $V_v$ is a $K_v$-vector subspace represented by an element $\bv_v \in \wedge^* K_v $. So $\tau(V)$ is just the sum of the expansion rates of $V_v$ in $K_v^d$ over all places $v \in S$. For a subset $M\subset\GL_d(K_S)$ we also define, 
\[ \tau_M(V) = \max_{L\in M} \tau(LV).\]
As in $(\ref{taulimit4})$ above, we see that if $V$ represented by $\bv \in \wedge^* K_S^d$,
\[ \tau_M(V) = \max\{ I(a);\ I=(I_v)_{v \in S}, \exists L \in M, \forall v \in S, (L^{(v)} \bv^{(v)})_{I_v} \neq 0\},\]
where  $I_v \subset [1,d]$, $|I_v|=\dim V_v$, $I(a)=\sum_{v\in S} I_v(a)$ and $I_v(a)=\sum_{i \in I_v} \log |a_i^{(v)}|$, and $\bw_I$ is defined by the expression $\bw=\sum_{|I|=k} \bw_I e_I$, where $e_I=e_{i_1}\wedge \ldots \wedge e_{i_k}$ when $I=\{i_1,\ldots,i_k\}$.

We will say that $M$ is Pl\"ucker irreducible if its projection to $\GL_d(K_v)$ is Pl\"ucker irreducible for each $v \in S$. Under this assumption we see by the same argument as in Lemma \ref{tausm} that $\tau_M$ is submodular on the set of all $K_S$-submodules of $K_S^d$. 

If we restrict $\tau_M$ to the set of $K$-subspaces of $K^d$, i.e. the grassmannian $\Grass(K^d)$, then we thus obtain a well-defined notion of Harder-Narasimhan filtration, Grayson polygon and slopes. This gives what we will call the rational Grayson polygon $\mathcal{G}^K$. And a $K$-linear space $V$ will be $M$-semistable for the semigroup $(a_t)$ if for every $K$-subspace $W<V$,
\[ \frac{\tau_M(W)}{\dim W} \geq \frac{\tau_M(V)}{\dim V}.\]

But we may also consider $\tau_M$ as a function on the set of all $K_S$-submodules of $K_S^d$. Since the dimension of each projection to $K_v$ may not be the same for all $v$, we use the following definition for the dimension of a submodule $V=\prod_v V_v$:
$$\dim V= \frac{1}{|S|}\sum_{v \in S} \dim_{K_v} V_v$$
Then $\dim V$ is a modular function on the ``full grassmannian'', i.e. the set of all $K_S$-submodules of $K_S^d$. Thus Proposition \ref{h-n} and its proof are still valid and we obtain a ``full'' Harder-Narasimhan filtration and a ``full'' Grayson polygon $\mathcal{G}^{K_S}$, whose nodes now have $x$-coordinates in $\frac{1}{|S|}\N$.

\subsection{Inheritance principle and proofs}

In this section we discuss the proof of Theorems \ref{subspacenf} and \ref{parametricnf}. 
A basic ingredient is the $S$-arithmetic version of Minkowski's second theorem, which without paying attention to numerical constants, takes the following form:

\begin{theorem}[Minkowski's second theorem]\label{minkowskisecond}
Let $\Delta$ be a sublattice in $K_S^d$ as in $(\ref{subla})$. Then
\[ c(\bx_1\wedge \ldots \wedge \bx_k)^{\frac{1}{[K:\Q]}} \leq \lambda_1(\Delta)\cdot\ldots\cdot\lambda_k(\Delta) \ll c(\bx_1\wedge \ldots \wedge \bx_k)^{\frac{1}{[K:\Q]}},\]
where the constant involved in the Vinogradov notation $\ll$ depends only on $K$, $S$ and $d$, not on $\Delta$.
\end{theorem}

The content $c(\bx_1\wedge \ldots \wedge \bx_k)$ is proportional to the covolume of $\Delta$ in its $K_S$-span. See \cite[Theorem~C.2.11, page 611]{bombierigubler} for a proof when $S$ has no non-archimedean places and \cite{kleinbock-shi-tomanov} in the general case with the caveat that the normalizations used in the latter paper differ from ours especially at the complex place, leading to a slightly different definition of the successive minima.

The derivation of Theorem \ref{subspacenf} from Theorem \ref{parametricnf} works verbatim as that of Theorem \ref{subspace-main} from Theorem \ref{parametrici} given in \S \ref{proof-sub}. One replaces $d$ with $d|S|$ and treats all linear forms $L_i^{(v)}$ on an equal footing. The constant $D$ needs to be increased appropriately and at non-archimedean places $v$ the exponential used in the definition of the flow will be replaced by a power of a uniformizer $\pi_v$ of $K_v$, namely $a^{(v)}_t=\diag(\pi_v^{n_{1,v}t},\ldots,\pi_v^{n_{d,v} t})$ for integers $n_{i,v}$. One needs also to recall that in view of $(\ref{diri})$, for each $T>0$ there are only finitely many classes of $\bx \in \mathcal{O}_{K,S}^d$ with $c(\bx) \leq T$, so we may assume that $c(\bx)$ is large. The number $r$ of distinct $V_i$ obtained is similarly bounded by $b(2^{d|S|})$ as in Lemma \ref{values}.

Now the proof of Theorem \ref{parametricnf} is again verbatim as that of Theorem \ref{parametrici}, treating all $L_{i}^{(v)}$ on an equal footing and keeping the argument unchanged. One needs to invoke the $S$-arithmetic subspace theorem (in the form of Theorem \ref{subspacenf} for a single point, or as \cite[7.2.5]{bombierigubler}) in place of the ordinary subspace theorem. For the dynamical ingredient at the end, one applies instead the following generalization of Theorem \ref{rate}.

\begin{theorem}[Inheritance principle]
\label{inheritancenf}
Let $K$ be a number field, $S$ a finite set of places containing all archimedean places, and $d$ in $\N$.
Let $(a_t)$ be a diagonal one-parameter semigroup in $\GL_d(K_S)$, and $\mu$ a good measure on $\GL_d(K_S)$ with $\mathcal{H}(\mu)=\prod_{v\in S} \mathcal{H}(\mu_v)$ the Pl\"ucker closure of its support. For  $k \in [1,d]$ let $$\mu_k(L):=\liminf_{t\to\infty} \frac{1}{t} \sum_{i\leq k} \log\lambda_i(a_tL\cO_{K,S}^d)$$
for each $L \in \GL_d(K_S)$. Then, for $\mu$-almost every $L$,
\[ \mu_k(L)= \sup_{L'\in \mathcal{H}(\mu)} \mu_k(L').\]
\end{theorem}

\begin{proof}[Sketch of proof]When $k=1$ the proof of this result is identical to the proof we gave of Theorem \ref{rate} with the following adjustments. Theorem \ref{klw} was extended to this context by Kleinbock-Tomanov in \cite[\S 8.4]{kleinbock-tomanov} (technically speaking only for $K=\Q$, but the general case is entirely analogous).  The norm there and in the definition (\ref{cbeta}) of $\beta$ must be replaced by the content of the corresponding vector. Theorem \ref{minkowskisecond} must be used in place of the original Minkowski theorem to prove the first claim. Also $(\ref{inZar})$  continues to hold with the content in place of the norm for subsets of $\GL_d(K_S)$ that are cartesian products of subsets of $\GL_d(K_v)$ for $v \in S$, which is the case for $\Supp(\mu)$ by definition of a good measure. So the second claim also holds. The case $k>1$ is analogous, but as in the proof of Prop. \ref{pol-comp}, one needs to use the refined quantitative non-divergence estimate proved in \cite[Chapter 6]{saxce-hdr} and \cite[Theorem 5.3]{lmms} instead of Theorem \ref{klw}.
\end{proof}




We end this section by stating the analogue of Proposition \ref{pol-comp} in the $S$-arithmetic context. The following describes  what is left of Theorem \ref{parametricnf} if we remove the assumption that the Pl\"ucker closure of the good measure $\mu$ is defined over $\QQ$. In \S \ref{HS-K} we have defined two Grayson polygons: the rational one $\mathcal{G}^{K}$ coming from $\Grass(K^d)$ and the full one $\mathcal{G}^{K_S}$ coming from the full grassmannian of all $K_S$-submodules. And of course we have as before the polygon $\mathcal{G}_\mu$ with nodes $(k,\mu_k)$, $k \in [1,d]$, where $\mu_k$ is the $\mu$-almost sure value of $\mu_k(L)$ provided by Theorem \ref{inheritancenf} and the polygon $\mathcal{G}_\mu^{sup}$  with nodes $(k,\mu^{sup}_k)$, where $\mu^{sup}_k$ is the supremum of $\mu^{sup}_k(L)$ over $L \in M$ and $\mu^{sup}_k(L)$ is defined as $\mu_k(L)$ with a limsup in place of liminf. 

\begin{proposition}[$S$-arithmetic sandwich theorem]
The polygons $\mathcal{G}_\mu, \mathcal{G}_{\mu}^{sup}$ lie in between $\mathcal{G}^K$ and $\mathcal{G}^{K_S}$, i.e. $$\mathcal{G}^{K_S} \leq \mathcal{G}_\mu \leq \mathcal{G}_{\mu}^{sup} \leq \mathcal{G}^{K}.$$
\end{proposition}

Again the proof is mutatis mutandis that of Proposition \ref{pol-comp}.

\section{Examples and Applications}\label{app}

In this last section we present a number of applications of the main theorem. 

\subsection{Sprindzuk conjecture}\label{sprin}

We begin by the demonstration of how the main result of Kleinbock and Margulis \cite[Conjectures H1, H2]{km}, namely the Sprindzuk conjecture, can easily be deduced from Theorem \ref{subspace-main}. Let us recall this result. For $\bq \in \Z^d$ we define $\Pi_+(\bq):=\prod_1^d |q_i|_+$, where $|x|_+:=\max\{1,|x|\}$ for all $x \in \R$. A point $y \in \R^n$ is said to be very well multiplicatively approximable (or VWMA for short) if for some $\eps>0$ there are infinitely many $\bq \in \Z^d$ such that \begin{equation}\label{vwm-def}|p+\bq \cdot y| \cdot \Pi_+(\bq)  \leq \Pi_+(\bq)^{-\eps}.\end{equation} A manifold $M \subset \R^d$ is said to be \emph{strongly extremal} if Lebesgue almost every point on $M$ is not VWMA.

\begin{theorem}[Kleinbock-Margulis \cite{km}]\label{km-thm} Let  $U \subset \R^n$ be a connected open set and $f_1,\ldots,f_{d}:U \to \R$ be real analytic functions, which together with $1$, are linearly independent over $\R$. Write $f=(f_1,\ldots,f_d)$. Then $M:=\{f(x), x \in U\}$ is strongly extremal.
\end{theorem}

As often with applications of the subspace theorem, proofs proceed by induction on dimension. The induction hypothesis will be as follows: if $g=(g_0,\ldots,g_d)$ is a tuple of linearly independent analytic functions on $U$ and $b_1,\ldots, b_d \in \Z^d$ are linearly independent, then for almost every $x \in U$ and every $\eps>0$, there are only finitely many solutions $\bv \in \Z^{d+1}$ to the inequalities 
\begin{equation}\label{km-sub}0 < \prod_0^d |L_i(\bv)| \leq \frac{1}{\|\bv\|^\eps},\end{equation} where $\|\bv\|=\max|v_i|$,  $L_0(\bv)=g(x) \cdot \bv$ and $L_i(\bv)=b_i \cdot \bv$ for $i \ge 1$. 
It is straightforward that this statement implies Theorem \ref{km-thm}, by letting $\bv=(p,\bq)$, $g=(1,f)$ and $b_1,\ldots,b_d$ the standard basis of $\Z^d$. 

Let $\phi(x)$ be the matrix whose rows are $L_0/g_0,L_1,\ldots,L_d$. The linear independence assumption implies that $\phi(x) \in \GL_{d+1}(\R)$ on an open subset $U' \subset U$, and without loss of generality we may assume that $U'=U$. The equations defining the Pl\"ucker closure of $\phi(U)$ are linear combinations of $k\times k$ minors of $\phi(x)$. Since those are linear combinations of the $g_i/g_0$, the Pl\"ucker closure of $\phi(U)$ is the Pl\"ucker closure of the set of all matrices in $\GL_{d+1}(\R)$ whose first row is $(1,y)$ with $y\in \R^d$ arbitrary and whose other rows are $L_1,\ldots,L_d$. So it is defined over $\Q$.


We may thus apply Theorem \ref{subspace-main} and conclude that  there is a finite number of proper rational hyperplanes $V$, such that for almost every $x$, the large enough solutions  $\bv$ to $(\ref{km-sub})$  are contained in some $V$.

Now pick one such $V$, and consider the restriction of $L_0$ to $V$. For $\bv \in V$ we can write $\bv=\sum_1^d k_iv_i$ for $\bk=(k_1,\ldots,k_d) \in \Z^d$, where $v_1,\ldots,v_d$ is a basis of $V \cap \Z^{d+1}$. Set $L'_0(\bk):=L_{0|V}(\bv)=\bv \cdot g = \bk \cdot h$, where $h=(h_1,\ldots,h_d)$ and $h_i=v_i \cdot g$. And for $i \ge 1$, $L'_i(\bk):=L_{i|V}(\bv)=\bv \cdot b_i = \bk \cdot c_i$, where $c_i=(b_i \cdot v_1,\ldots,b_i \cdot v_d)$. The $L_i$ are linearly independent, so the $(L'_i)_0^d$ have rank at least $d$ on $\R^d$, and thus $(c_1,\ldots,c_d)$ has rank at least $d-1$. Up to reordering, we may assume that $c_1,\ldots,c_{d-1}$ are linearly independent. Similarly we see that $h_1,\ldots,h_d$ are linearly independent.

Finally observe that a solution $\bv \in V$ to  $(\ref{km-sub})$ yields a $\bk \in \Z^d$ such that
$$ 0< \prod_0^d |L'_i(\bk)| \ll \frac{1}{\|\bk\|^\eps}.$$
But $L'_d(\bk)=\bv \cdot b_d \in \Z \setminus\{0\}$. Thus $0<\prod_0^{d-1} |L'_i(\bk)| \ll \frac{1}{\|\bk\|^\eps}$ and  thus, by induction hypothesis, $\bk$ belongs to a finite set of points. Hence so does $\bv$. This ends the proof.

\subsection{Ridout's theorem for manifolds}
Ridout's theorem \cite{bombierigubler,zannier-lectures} is an extension of Roth's theorem where $p$-adic places are allowed. This improves the exponent in Roth's theorem from $2$ to $1$ in case the rational approximations have denominators with prime factorization in a fixed subset.  In this paragraph,  we present one possible similar variant of the Kleinbock-Margulis theorem (Theorem \ref{km-thm}). This will rely on the $S$-arithmetic subspace theorem for manifolds (Theorem \ref{subspacenf}). 

\begin{theorem}\label{ridout-manifold} Let $S$ be a finite set of primes. Let $U \subset \R^n$ be a connected  open subset and $f_1,\ldots,f_d:U \to \R$ be real analytic functions, which together with $1$ are linearly independent over $\R$. Write $f=(f_1,\ldots,f_d)$. Then for every $\eps>0$ and for Lebesgue almost every $u \in U$, we have \begin{equation}\label{rid}\max_{1\leq i \leq d} |qf_i(u) - p_i| \ge q^{-\eps}\end{equation} for all $(p_1,\ldots,p_d,q) \in \Z^{d+1}$ with all prime factors of $q$ in $S$, except for finitely many exceptions. 
\end{theorem}

This is in contrast with the same result \cite{km} without the restriction on the denominator $q$, where the right-hand side of $(\ref{rid})$ needs to be replaced by the weaker bound $q^{-\frac{1}{d}-\eps}$.  We have chosen simultaneous approximation for a change, but a similar statement with similar proof holds also for linear forms. Besides the theorem holds under the weaker assumption that the subspace of $\R^{d+1}$ spanned by all vectors $(1,f_1(x),\ldots,f_d(x))$, $x \in U$, is defined over $\QQ$.

\begin{proof}[Proof sketch]The proof is a straightforward modification of the one we gave of Theorem \ref{km-thm} in \S \ref{sprin}. We define linear forms on $\R^{d+1}$,   $L^{(\infty)}_{u,i}$ is $x_1$ if $i=1$ and $x_1f_i(u) - x_i$ if $i>1$. And if $p\in S$,  $L^{(p)}_{u,i}$ is constant equal to $x_i$. And we note that if $\bx:=(q,p_1,\ldots,p_d) \in \Z^{d+1}$ contradicts $(\ref{rid})$, then $$\prod_{1 \leq i \leq d+1}  \prod_{p \in \{\infty\} \cup S} |L_{u,i}^{(p)}(\bx)|_p \leq H(\bx)^{-d\eps}.$$ So Theorem \ref{subspacenf} applies and if $q$ is large enough, $\bx$ belongs to a finite family of proper rational subspaces. This allows us to use induction after restricting the linear forms to one of these subspaces, as in the proof of  Theorem \ref{km-thm}. The proof is left to the reader.
\end{proof}

\subsection{Submanifolds of matrices: extremality and inheritance}  \label{subsec:ede}
In this paragraph, we describe the extension of the Kleinbock-Margulis theorem to the case of subsmanifolds of matrices and we show how to recover from the subspace theorem for manifolds one of the main result of \cite{abrs2}, which is a criterion for extremality in terms of so-called constraining pencils and an explicit computation of the exponent.

In what follows, $E$ and $V$ are two finite-dimensional real vector spaces with a $\Q$-structure. We fix a lattice $\Delta$ in $V$, which defines the rational structure. For $x$ in $\Hom(V,E)$, we define
\[
\beta(x) = \inf\{\beta>0 \ |\ \exists c>0:\ \forall v\in \Delta\setminus\{0\},\
\|x(v)\|\geq c\|v\|^{-\beta}\}.
\]
Note that $\beta(x)$ only depends on the subspace $\ker x$ in $V$.


Given a subspace $W$ in $V$ and an integer $r$, we define the \emph{pencil of endomorphisms} $\cP_{W,r}$ by
\[ \cP_{W,r} = \{ x\in\Hom(V,E) \ |\ \dim x(W) \leq r\}.\]
We say that $\cP_{W,r}$ is constraining if $ \dim W/r < \dim V / \dim E$ and rational if $W$ is so. 
In the case of algebraic sets defined over $\Q$, the following theorem was proved in joint work with Menny Aka and Lior Rosenzweig \cite[Theorem~1.2]{abrs2}.
The approach taken here yields a different proof of that result, and allows to generalize it to subsets defined over $\QQ$.

\begin{theorem}[Diophantine exponent for submanifolds of matrices]
\label{extremalityexponent} Let $U \subset \R^n$ a connected open set and $\phi:U \to \Hom(V,E)$ an analytic map.
Assume that the Zariski closure of  $M=\phi(U)$ is defined over $\overline{\Q}$.
Then, for Lebesgue almost every $u$ in $U$, setting $x=\phi(u)$,
\[ \beta(x) = \max \left\{ \frac{\dim W}{r}-1 \ ;\ W\ \mbox{rational subspace such that}\  M\subset\cP_{W,r}\right\}.\]
\end{theorem}

\begin{remark}For any sublattice $\Delta'\leq\Delta$, one may define
\[ r(\Delta')=\max\left\{\dim\Span_{\R}x(\Delta');\ x\in M\right\}.\]
Then, the formula in the above theorem is simply, for almost every $x$ in $M$,
\[ \beta(x) =\beta:= \max \left\{ \frac{\rk\Delta'}{r(\Delta')}-1 \ ;\ \Delta'\leq\Delta\right\}.\]
\end{remark}

%

\begin{proof}[Proof of Theorem~\ref{extremalityexponent}]
%
We prove the theorem by induction on $d=\dim V$, using the subspace theorem.\\
\underline{$d=1$.} The result is clear because, for every $x$, the subgroup $x(\Delta)$ is a discrete subgroup of $E$.\\
\underline{$d-1\to d$.}
Suppose the result has been proven for $d-1\geq 1$.
Let $m=\dim E$ and $d=\dim V$.
Fixing bases for $E$ and $\Delta$, we identify $\Hom(V,E)$ with $m\times d$ matrices. 
Given $x$ in $M$, we denote by $x_i$, $1\leq i\leq d$, its columns, which are vectors in $E$.
The rank of $x$ is almost everywhere constant, and taking a coordinate projection if necessary, we assume that this rank is $m \leq d$ and that for almost every $x$ in $M$, 
the vector space $E$ is spanned by the first $m$ columns $x_i$, $1\leq i\leq m$.
We want to show that, for almost every $x$ in $M$, for all $\eps>0$,
the set of inequalities
\begin{equation}\label{limmooo}
|v_1x_{i1}+v_2x_{i2}+\dots+v_dx_{id}|\leq \|v\|^{-\beta-\epsilon}, \quad 1\leq i\leq m,
\end{equation}
has only finitely many solutions $\bv=(v_1,\dots,v_d)$ in $\mathbb{Z}^d$.
For $1\leq i\leq m$, define a linear form on $V\simeq\mathbb{R}^d$ by
$$L_i(\bv)=v_1x_{i1}+v_2x_{i2}+\dots+v_dx_{id},$$
and for $m< i\leq d$,
$$L_i(\bv)=v_i.$$
Since $(x_i)_{1\leq i\leq m}$ spans $E$, the family of linear forms $(L_i)_{1\leq i\leq d}$ is linearly independent.
Moreover, as $\beta \geq \frac{\rk\Delta}{r(\Delta)}-1=\frac{d-m}{m}$, condition~(\ref{limmooo}) certainly implies
$$\prod_{i=1}^d|L_i(\bv)|\leq \|v\|^{-\epsilon}.$$
We may therefore apply Theorem \ref{subspace-main}: there exists a finite family $V_1,\dots,V_h$ of hyperplanes in $\mathbb{Q}^k$ such that, for almost every $x$ in $M$, the integer solutions to (\ref{limmooo}) all lie in the union $V_1\cup\dots\cup V_h$ except a finite number of them.
It now suffices to check that in each $V_i$, there can be only finitely many solutions.
This follows from the induction hypothesis applied to  $V'=V_i$, $\Delta'=V_i\cap\Delta$ and to the manifold $M'$ image of $M$ under restriction to $V'$.
%

The converse inequality $\beta(x)\geq \beta$ is true for all $x$ in $M$, as is easily seen using the classical Dirichlet argument.
\end{proof}

It is also worth observing the following relation between the notions of extremality and semistability.  In \cite{km, kleinbock-margulis-wang} an analytic submanifold $M$ of $M_{m,n}(\R)$ is said to be extremal if for almost every $Y\in M$ and every $\eps>0$ there are only finitely many vectors $q \in \Z^n$ and $p \in \Z^m$ such that $\|Yq-p\| \le \|q\|^{-(n/m+\eps)}$. Now consider the matrix 
$$L_Y:=\begin{pmatrix}
I_m & Y\\
 & I_n
\end{pmatrix}$$
The image $\widetilde{M}$ of $M$ under $Y \mapsto L_Y$ defines a submanifold of $\GL_{n+m}(\R)$. Using Proposition \ref{pol-comp} and Theorem \ref{parametrici}, it is then easy to see that:

\begin{proposition}[extremality vs. semistability]\label{extvssemi} If $\R^{n+m}$ is $\widetilde{M}$-semistable with respect to the unimodular flow: $$a_t=(e^{tn},\ldots, e^{tn}, e^{-tm},\ldots,e^{-tm}),$$ then $M$ is extremal. If the Zariski (or Pl\"ucker) closure of $M$ is defined over $\QQ$, then the almost sure diophantine exponent $\beta$ from Theorem \ref{extremalityexponent} can be read off the rational Grayson polygon of $\widetilde{M}$ by the formula $$1+\beta = \frac{n+m}{\gamma + m},$$ where $\gamma$ is the smallest slope of the rational Grayson polygon. In particular $\Q^{n+m}$ is $\widetilde{M}$-semistable if and only if $M$ is extremal. 
\end{proposition}


\subsection{Multiplicative approximation and strong extremality}\label{multi-strong}


Following the suggestion of Baker \cite[page 96]{baker_tnt} to study the multiplicative diophantine properties of the Mahler curve, Kleinbock and Margulis introduced the notion of \emph{strong extremality} for manifolds in $\R^n$.
This was later generalized in \cite{kleinbock-margulis-wang,kmb} to the context of diophantine approximation on matrices, but in that generalized setting the optimal criterion for strong extremality remained to be found \cite{kleinbock-margulis-wang,kmb}. The method of the present paper can be used to answer this problem. Below we apply Theorem~\ref{subspace-main} and give a complete solution when the Zariski closure of the manifold is  defined over $\QQ$.

\bigskip

Let $m$ and $n$ be two positive integers, and $M_{m,n}(\R)$ the space of $m\times n$ matrices with real entries.
Following Kleinbock and Margulis \cite{km}, we say that a matrix $Y\in M_{m,n}(\R)$ is \emph{very well multiplicatively approximable} (VWMA) if 
there exists $\eps>0$ such that the inequality
\[ \prod_{i=1}^m |Y_i\bq - p_i| \leq \prod_{j=1}^n |q_j|_+^{-1-\eps} \]
has infinitely many solutions $(\bp,\bq)\in\Z^m\times\Z^n$.
In the above inequality, $Y_i$ denotes the $i$-th row of $Y$, for $i=1,\dots,m$, and $|q|_+=\max(|q|,1)$.
More generally, we define as in \cite[\S1.4]{dfsu} the \emph{multiplicative diophantine exponent} of $Y\in M_{m,n}(\R)$ as
\[ \omega_\times(Y) =
\sup\left\{\omega>0;\ 
\prod_{i=1}^m |Y_i\bq - p_i| \leq \prod_{j=1}^n |q_j|_+^{-\omega}\ \mbox{for infinitely many}\ (\bp,\bq)\right\}. \]
With this definition, we see that a matrix $Y$ is VWMA if and only if $\omega_\times(Y)>1$.
Using the Borel-Cantelli lemma, it is not difficult to check that for the Lebesgue measure, almost every $Y$ in $M_{m,n}(\R)$ satisfies $\omega_\times(Y)=1$.
It is therefore natural to ask what other measures $\mu$ on $M_{m,n}(\R)$ satisfy this property. 

\smallskip

We now set up some notation to formulate our criterion for stong extremality.
Given a matrix $Y$ in $M_{m,n}(\R)$, with rows $Y_1,\dots,Y_m$, we let
\[ L_Y= \begin{pmatrix} -I & Y\\ 0 & I\end{pmatrix},\]
and denote by $L_i$, $i=1,\dots,m+n$ the linear forms on $\R^{m+n}$ given by the rows of the matrix $L_Y$:
\[ \left\{
\begin{array}{ll}
L_i(\bp,\bq) = Y_i\bq -p_i  & \mbox{for}\ i\in\{1,\dots,m\}\\
L_i(\bp,\bq) = q_{i-m} & \mbox{for}\ i\in\{m+1,\dots,m+n\}.
\end{array}
\right.\]
If $W$ is a linear subspace of $\R^{n+m}$, and $I$ a non-empty subset of $\{1,\dots,m+n\}$, we let
\[ s_{I,W} = \rk (L_i|_W)_{i\in I}. \]

\begin{definition}[Multiplicative pencils]
Let $I,J$ be proper subsets of $\{1,\dots,m+n\}$ such that $I\subset\{1,\dots,m\}\subset J$ and $r,s$ non-negative integers.
Given a subspace $W\leq\R^{n+m}$, we define a subvariety of endomorphisms $\cP_{I,J,r,s,W}\subset M_{m,n}(\R)$ by
\[ \cP_{I,J,r,s,W} = \{ Y\in M_{m,n}(\R) \,|\, s_{I,W}\leq r \ \mbox{and}\  s_{J,W}\leq s \}. \]
\end{definition}

To justify the relevance of this definition to our problem, we start by an easy proposition, which is a consequence of Minkowski's first theorem or Dirichlet's pigeonhole principle.

\begin{proposition}[Dirichlet's principle]
\label{dirichlet}
Fix $Y\in M_{m,n}(\R)$, and denote by $L_i$, $i=1,\dots,m+n$, the rows of the matrix $L_Y$.
Assume that $W$ is a $k$-dimensional rational subspace of $\R^{m+n}$
and that $\emptyset\neq I\subset\{1,\dots,m\}\subset J\subsetneq \{1,\dots,m+n\}$ are such that
\[ r=\rk(L_i|_W)_{i\in I} \quad\mbox{and}\quad s=\rk(L_i|_W)_{i\in J}.\]
Then,
\[ \omega_\times(Y) \geq \frac{(k-s)|I|}{r(n+m-|J|)}.\]
\end{proposition}

\begin{remark}
By convention, if $r=0$, the ratio is equal to $+\infty$.
Note that one can always take $W=\R^{m+n}$, $I=J=\{1,\dots,m\}$, and $r=s=m$, in which case the ratio is equal to $1$.
\end{remark}

\begin{proof}
If $r=0$, then one must have $Y_i\bq-p_i=0$ for some integer vector $(\bp,\bq)$ in $\Z^{m+n}$.
It is then clear that $\omega_\times(Y)=\infty$.
So we may assume that $r\neq 0$.
Since by definition $\omega_\times(Y)\geq 0$, we may also assume that $s<k$, otherwise there is nothing to prove.

Assuming that $Y\in \cP_{I,J,r,s,W}$, we shall prove that there exists a constant $C>0$ depending only on $Y$ and $W$ such that the inequality
\[ \prod_{i=1}^m |L_i(v)| \leq C \prod_{j=1}^n |L_{m+j}(v)|^{-\frac{(k-s)|I|}{r(n+m-|J|)}}\]
has infinitely many solutions $v$ in $W$. This will yield the desired lower bound on  $\omega_\times(Y)$.
Let $Q>0$ be a large parameter.
Pick $i_1,\dots,i_r$ in $I$ such that $L_{i_1}|_W,\dots,L_{i_r}|_W$ are linearly independent, pick $i_{r+1},\dots,i_s$ in $J$ such that $L_{i_1}|_W,\dots,L_{i_s}|_W$ are linearly independent, and complete with $i_{s+1},\dots,i_k$ such that $L_{i_1}|_W,\dots,L_{i_k}|_W$ are linearly independent.
The symmetric convex body in $W$ defined by
\[ \left\{
\begin{array}{ll}
|L_{i_\ell}(v)| \leq Q^{-(k-s)} & \mbox{for}\ 1\leq\ell\leq r\\
|L_{i_\ell}(v)| \ll 1 & \mbox{for}\ r<\ell\leq s\\
|L_{i_\ell}(v)| \leq Q^{r} & \mbox{for}\ s<\ell\leq k
\end{array}
\right.\]
has volume $\gg 1$ and therefore, by Minkowski's first theorem, it contains a non-zero point $v$ in $W\cap\Z^{n+m}$.
By our choice of the indices $i_\ell$, $1\leq \ell\leq k$, such a point satisfies
\[ \left\{
\begin{array}{ll}
|L_i(v)| \leq Q^{-(k-s)} & \mbox{for}\ i\in I\\
|L_i(v)| \ll 1 & \mbox{for}\ i\in J\setminus I\\
|L_i(v)| \leq Q^{r} & \mbox{for}\ i\not\in J
\end{array}
\right.\]
and therefore
\[ \prod_{i=1}^m |L_i(v)| \ll Q^{-|I|(k-s)}
\ll \prod_{j=1}^n |L_{m+j}(v)|^{-\frac{(k-s)|I|}{r(n+m-|J|)}}.\]
\end{proof}

Now, as in the introduction, let $M=\phi(U)$ be a connected analytic submanifold of $M_{m,n}(\R)$ endowed with the measure $\mu$ equal to the push-forward of the Lebesgue measure under the analytic map $\phi:U \to M_{m,n}(\R)$. When the Zariski closure of $M$ is defined over $\QQ$ -- for example when $\phi$ is given by a polynomial map with coefficients in $\QQ$ -- the above proposition actually provides a formula for $\omega_\times(Y)$, when $Y$ is a $\mu$-generic point of $M$. In other words:

\begin{theorem}[Formula for the multiplicative exponent]
\label{omegaformula}
Assume that the Zariski closure of $M$ is defined over $\QQ$. Then 
 for $\mu$-almost every $Y$ in $M$,
\[
\omega_\times(Y)
= \max_{\substack{M\subset\cP_{I,J,r,s,W}\\W\ \mathrm{rational}}} \frac{(\dim W-s)|I|}{r(m+n-|J|)}.\]
\end{theorem}

We stated the result for analytic submanifolds for convenience, but it holds with the same proof for all good measures $\mu$ in the sense of \S \ref{mgm} provided the Zariski closure of the support of $\mu$ is defined over $\QQ$. 
We shall say that a multiplicative pencil $\cP_{I,J,r,s,W}$ is \emph{constraining} if it satisfies
\[ \frac{(\dim W-s)|I|}{r(m+n-|J|)} > 1.\]
Our criterion\footnote{We were informed by David Simmons \cite{simmons} that he and Tushar Das also obtained a similar  criterion for strong extremality of submanifolds defined over $\R$.} for strong extremality immediately follows from the above formula.

\begin{corollary}[Criterion for strong extremality]\label{criterion}
If $M$ is an analytic submanifold of $M_{m,n}(\R)$ whose Zariski closure is defined over $\QQ$.
Then $M$ is strongly extremal if and only if it is not contained in any rational constraining pencil.
\end{corollary}

The proof of Theorem~\ref{omegaformula} is inspired by Schmidt's proof of an analogous result on products of linear forms \cite[\S12, page~242]{schmidt_da}.

\begin{proof}[Proof of Theorem~\ref{omegaformula}]
Let $\omega>0$ be such that for $Y$ in a set of positive measure in $M$, the inequality
\begin{equation}\label{mult}
\prod_{i=1}^m|Y_i\bq-p_i| \leq \prod_{j=1}^m|q_j|_+^{-\omega}
\end{equation}
has infinitely many solutions $(\bp,\bq)$ in $\Z^{m+n}$.
We want to show that $M$ is included in a pencil $\cP_{I,J,r,s,W}$ such that
\[ \frac{(\dim W-s)|I|}{r(m+n-|J|)} \geq \omega.\]
Let $W$ be a rational subspace of minimal dimension $k$ containing infinitely many solutions to \eqref{mult}.
If $k=1$, then there must exist $i\in\{1,\dots,m\}$ and $(\bp,\bq)\in\Z^{m+n}$ such that $Y_i\bq-p_i=0$, and one can take $I=\{i\}$, $J=\{1,\dots,m\}$, $r=0$, and $s=1$.
So we assume $k\geq 2$.
Reordering the indices if necessary, we may assume that $W$ contains infinitely many solutions to \eqref{mult} satisfying
\begin{equation}\label{order1}
0 < |Y_1\bq -p_1| \leq \dots \leq |Y_m\bq-p_m|
\end{equation}
and
\begin{equation}\label{order2}
|q_1|\leq \dots\leq |q_n|.
\end{equation}
Define $i_1<i_2<\dots<i_f\leq m$ inductively so that each $i_\ell$ is minimal such that $\rk(L_{i_1}|_W,\dots,L_{i_\ell}|_W)=\ell$, and then $m<j_1<\dots<j_g\leq m+n$ such that $j_\ell$ is minimal such that $\rk(L_{i_1}|_W,\dots,L_{i_f}|_W,L_{j_1}|_W,\dots,L_{j_\ell}|_W)=f+\ell$.
Note that in our notation, $f+g=k=\dim W$.

Given a large solution $\bv=(\bp,\bq)$ to \eqref{mult} in $W$, choose numbers $c_1,\dots, c_f$ such that
\[ |L_{i_\ell}(\bv)|\asymp \|\bv\|^{-c_\ell} \quad \mbox{for}\ \ell=1,\dots,f\]
and $d_1, \dots, d_g$ such that
\[ |L_{j_\ell}(\bv)|\asymp \|\bv\|^{d_\ell} \quad \mbox{for}\ \ell=1,\dots,g.\]
By \eqref{order1}, one has $c_1\geq\dots\geq c_f\geq 0$.
By \eqref{order2} and the fact that one always has $|L_{j_\ell}(\bv)|\ll\|\bv\|$, one finds
$0\leq d_1\leq\dots\leq d_g\leq 1$.
Moreover, by minimality of $W$, the subspace theorem applied in $W$ to the set of linear forms 
\[ L_{i_1}|_W,\dots,L_{i_f}|_W,L_{j_1}|_W,\dots,L_{j_g}|_W \]
shows that $-c_1+\dots -c_f+d_1+\dots+d_g\geq 0$.
In conclusion, one sees that the $k$-tuple $(c_1,\dots,c_f,d_1,\dots,d_g)$ belongs to the convex polytope $(P)$ defined by
\[ (P)\quad
\left\{
\begin{array}{l}
c_1\geq\dots\geq c_f\geq 0\\
0\leq d_1\leq\dots\leq d_g\leq 1\\
d_1+\dots+d_g\geq c_1+\dots+c_f.
\end{array}
\right.
\]
Now inequality \eqref{mult} implies
\[ c_1(i_2-i_1)+\dots+c_f(m-i_f) - \omega[d_1(j_1-m)+\dots+d_g(m+n-j_g)] \geq 0.\]
The linear map $f(\bc,\bd)=c_1(i_2-i_1)+\dots+c_f(m-i_f) - \omega[d_1(j_1-m)+\dots+d_g(m+n-j_g)]$ is non-negative on some point in the convex polytope $(P)$, so it must be non-negative on one of its vertices.
The polytope (P) has $gf$ vertices, given by
\[ p_{a,b}:\left\{
\begin{array}{l}
0=d_1=\dots=d_a<d_{a+1}=\dots=d_g=1\\
\frac{g-a}{b}=c_1=\dots=c_b>c_{b+1}=\dots=c_f=0
\end{array}
\right.
\quad\mbox{where}\
\left\{
\begin{array}{l}
a=0,\dots,g-1\\
b=1,\dots,f
\end{array}
\right.\]
Choose $(a,b)$ such that $f(p_{a,b})\geq 0$, and let
\[\left\{
\begin{array}{l}
I=\{1,\dots,i_b\}\\
J=\{1,\dots,j_a\}\\
r=b\\
s=f+a.
\end{array}
\right.\]
We then obtain
\[ 0 \leq f(p_{a,b})
 = \frac{g-a}{b}|I|-\omega(m+n-|J|),\]
whence
\[ \frac{(k-s)|I|}{r(m+n-|J|)} \geq \omega.\]
\end{proof}

\begin{remark}Strong extremality also relates to semistability in a similar way as in Proposition \ref{extvssemi}. We leave it to the reader to check that in the setting of Theorem \ref{omegaformula}, $M$ is strongly extremal if and only if $\Q^{m+n}$ is semistable for $\{L_Y, Y \in M\}$ with respect to all unimodular flows $a_t=(e^{tA_1},\ldots,e^{tA_{n+m}})$ with $A_1 \ge \ldots \ge A_m \ge 0 \ge A_{m+1} \ge \ldots \ge A_{m+n}$. See \cite{kleinbock-margulis-wang} where the relevance of  this family of flows for strong extremality was uncovered.
\end{remark}

\subsection{Roth's theorem for nilpotent Lie groups}
Diophantine approximation on nilpotent Lie groups was studied in \cite{abrs,abrs2}. In this paragraph we continue this study and  derive an analogue of Roth's theorem \cite{roth} in this context.  In doing so we extend Theorem \ref{extremalityexponent} to approximation with quasi-norms. 

\subsubsection*{Diophantine approximation in nilpotent Lie groups}

In this paragraph $G$ will denote a simply connected nilpotent Lie group, endowed with a left-invariant riemannian metric. We identify it with its Lie algebra $\g$ under the exponential map, so that its Haar measure is given by the Lebesgue measure on $\g$.
For a finite symmetric subset $S$ of $G$ and $\Gamma=\langle S \rangle$ the subgroup it generates, 
the diophantine exponent of $\Gamma$ in $G$ is
\[ \beta(\Gamma)=\inf \{ \beta \ |\
\exists c>0:\,\forall n\in\N^*,\,\forall x\in S^n\setminus\{1\},\,
d(x,1)\geq c n^{-\beta}\},\] 
where $S^n$ denotes the set of elements of $\Gamma$ that can be written as a product of at most $n$ elements of $S$.
Because $\Gamma$ is nilpotent, this definition does not depend on the choice of $S$.
The following theorem was proved in \cite[Theorem~7.4]{abrs2}.

\begin{theorem}[Existence of the exponent]
Let $G$ be a connected and simply connected real nilpotent Lie group endowed with a left-invariant riemannian metric $d$.
For each $k\geq 1$, there is $\beta_k\in[0,\infty]$ such that for almost every $k$-tuple $\bg=(g_1,\ldots,g_k) \in G^k$ with respect to the Haar measure, the subgroup $\Gamma_\bg$ generated by $g_1,\ldots,g_k$ satisfies
\[ \beta(\Gamma_{\bg}) = \beta_k.\]
\end{theorem}

We also showed in \cite{abrs2} that when $G$ is rational, i.e. when its Lie algebra $\g$ admits a basis with rational structure constants, then $\beta_k$ can be explicitly computed and in particular is rational. In fact $\beta_k=F(k)$ for some rational function $F \in \Q(X)$ when $k$ is large enough.
Using the subspace theorem for manifolds, Theorem \ref{subspace-main}, we can now show the following.

\begin{theorem}[Roth's theorem for nilpotent groups]
\label{rothnilpotent}
Let $G$ be a connected simply connected nilpotent Lie group and $k$ a positive integer.
Assume that the Lie algebra $\g$ of $G$ admits a basis with structure constants in a number field $K$.
Then $$\beta_k\in\Q.$$ Moreover, for every $k$-tuple $\bg \in \g(\overline{\Q})^k$ we have $\beta(\Gamma_\bg) \in \Q$. Furthermore, there is $C=C(\dim \g)>0$ and a countable union $\mathcal{U}$ of proper algebraic subsets of $\g^k$  defined over $K$ and of degree at most $C$, such that
\[ \beta(\Gamma_{\bg}) = \beta_k\] for every $\bg \in \g^k(\overline{\Q}) \setminus \mathcal{U}$.
\end{theorem}

The case when $k=2$ and $G=(\R,+)$ is exactly Roth's theorem. In this case $\beta_2=1$ and $\mathcal{U}$ is the family of lines in $\R^2$ with rational slopes.

The basic idea for the proof of Theorem~\ref{rothnilpotent}, developed in \cite{abrs2}, is to reduce the problem to a question of diophantine approximation on submanifolds.
For that, we introduce the free Lie algebra $\cF_k$ over $k$ generators $\sx_1,\dots,\sx_k$.
Inside $\cF_k$, the ideal of laws $\cL_{k,\g}$ on the Lie algebra $\g$ of $G$ is the set of elements $\sr$ in $\cF_k$ such that $\sr(X_1,\dots,X_k)=0$ for every $X_1,\dots,X_k$ in $\g$, and the ideal of rational laws $\cL_{k,\g,\Q}$ is the real span of the intersection of $\cL_{k,\g}$ with $\cF_k(\Q)$, the natural $\Q$-structure on $\cF_k$.
The Lie algebra $\cF_{k,\g,\Q}=\cF_k/\cL_{k,\g,\Q}$ has a graded structure
\[ \cF_{k,\g,\Q} = \bigoplus_{i=1}^s \cF_{k,\g,\Q}^{[i]},\]
where $\cF_{k,\g,\Q}^{[i]}$ is the homogeneous part of $\cF_{k,\g,\Q}$ consisting of brackets of degree $i$.
For $\sr=\sum \sr_i$ with $\sr_i\in \cF_{k,\g,\Q}^{[i]}$, we let \begin{equation}\label{quasinorm}|\sr|:=\max_{i=1,\ldots,s} \|\sr_i\|^{\frac{1}{i}},\end{equation} where $\|\cdot\|$ is a fixed norm on $\cF_{k,\g,\Q}$.
Endowed with this quasi-norm, the Lie algebra $\cF_{k,\g,\Q}$ is quasi-isometric to the group of word maps on $G$, endowed with the word metric \cite[Proposition~7.2]{abrs2}.
This yields the following characterization for the above diophantine exponent $\beta(\Gamma_{\bg})$ proved in \cite[Proposition~7.3]{abrs2}. We say that $\Gamma_{\bg}$ is relatively free in $G$ if the only relations satisfied by $\bg$ are the laws of $G$. This holds for all $\bg$ outside a countable union of proper algebraic subvarieties of bounded degree defined over $\Q$, and in particular for Lebesgue almost every $\bg \in G^k$.

\begin{proposition}
\label{wordbracket}
Let $G$ be a simply connected nilpotent Lie group, with Lie algebra $\g$.
Let $\bg=(e^{X_1},\dots,e^{X_k})$ be a $k$-tuple in $G$ such that $\Gamma_{\bg}$ is relatively free in $G$.
Then the exponent $\beta(\Gamma_{\bg})$ defined above is also the infimum of all $\beta>0$ such that
\begin{equation}\label{liediophantine}
\|\sr(X_1,\dots,X_k)\| \geq |\sr|^{-\beta}
\end{equation}
holds for all but finitely many $\sr\in\cF_{k,\g,\Q}(\Z)$.
\end{proposition}

In the next paragraph, we show how Theorem \ref{subspace-main} yields a formula for diophantine exponents defined with quasi-norms, as in the above proposition.

\subsubsection*{Weighted diophantine approximation}

This paragraph generalizes the results of \S\ref{subsec:ede} to diophantine approximation with quasi-norms, also called \emph{weighted diophantine approximation} (see e.g. \cite[\S 1.5]{k-m-w} and references therein).
Again, $E$ and $V$ are two finite-dimensional real vector spaces, and $\Delta$ is a lattice in $V$, defining a rational structure.

We fix a norm $\|\cdot\|$ on $E$.
On $V$, we measure the size of vectors using a quasi-norm $|\cdot|$ given by the formula
\begin{equation}\label{quasi-def} |v| = \max_{1\leq i\leq d} |\langle v,u_i^*\rangle|^{\frac{1}{\alpha_i}},\end{equation}
where $\alpha=(\alpha_1,\dots,\alpha_d)$ is a $d$-tuple of positive real numbers and $(u_i^*)_{1\leq i\leq d}$ a basis of $V^*$.
Given $x$ in $\Hom(V,E)$, define the diophantine exponent of $x$ by
\[ \beta_\alpha(x) = \inf\{\beta>0 \ |\ \exists c>0:\ \forall v\in\Delta,\
\|xv\| \geq c|v|^{-\beta}\}.\]


\begin{definition}[Growth rate for a quasi-norm]
Let $W$ be a linear subspace in $V$.
The \emph{growth rate} of balls in $W$ for the quasi-norm $|\cdot|$ is given by
\[ \alpha(W) = \lim_{R\to\infty} \frac{1}{\log R}\log\Vol\{v\in W\ |\ |v|\leq R\}.\]
\end{definition}

\begin{remark} It is not hard to see that this limit exists and equals $\sum_{i \in I_W} \alpha_i$ for a certain subset $I_W \subset [1,d]$. Indeed the restriction of $|\cdot|$ to $W$ is itself comparable up to multiplicative constants to a quasi-norm with exponents $\alpha_i$, $i \in I_W$, where $I_W=\{i_1,\ldots,i_k\}$ is defined as follows. Choose $i_1$ minimal such that the restriction of $u_{i_1}^*$ to $W$ is non-zero, then inductively  choose $i_j$ minimal such that the linear forms $u_{i_1}^*|_{W},\dots,u_{i_j}^*|_{W}$ are linearly independent.
\end{remark}

%

This gives us a lower bound for the diophantine exponent, using a standard Dirichlet type argument:

\begin{lemma}
Let $V$ and $|\cdot|$ be as above.
For any $x$ in $\Hom(V,E)$,
\[ \beta_\alpha(x) \geq \max\{ \frac{\alpha(W\cap\ker x)}{\dim W-\dim W\cap\ker x} \ ;\ 
W\leq V \ \mbox{rational subspace}\}.\]
\end{lemma}
\begin{proof}
Let $R>0$ be some large parameter.
The number of points $v$ in $W\cap\Delta$ such that $|v|\leq R$ is roughly $R^{\alpha(W)}$, and their images in $x(W)\simeq W/(W\cap\ker x)$ lie in a distorted ball of volume $O(R^{\alpha(W)-\alpha(W\cap\ker x)})$.
Comparing volumes, we find that balls of radius $\eps$ around those points cannot be disjoint if $\eps\gg R^{-\frac{\alpha(W\cap\ker x)}{\dim W-\dim W\cap\ker x}}$.
\end{proof}

It turns out that this lower bound is in fact attained almost everywhere on analytic submanifolds of $\Hom(V,E)$ whose Zariski closure is defined over $\QQ$.
This is the content of the next theorem. As earlier, we endow $M:=\phi(U)$ with the push-forward $\mu$ of the Lebesgue measure on the connected open subset $U \subset \R^d$ via the analytic map $\phi:U  \to \Hom(V,E)$.

\begin{theorem}[Diophantine exponent for quasi-norms]
\label{inhomogeneous}
Assume $V$ and $|\cdot|$ are as above, and that the Zariski closure $Zar(M)$ of $M$ is defined over $\overline{\Q}$.
Then, for almost every $x$ in $M$,
\begin{equation}\label{formulaforexp} \beta_\alpha(x) = \max\left\{\min_{y\in M}  \frac{\alpha(W\cap\ker y)}{\dim W-\dim W\cap\ker y} \ ;\ 
W\leq V \ \mbox{rational subspace}\right\}.\end{equation}
This equality is also true for every $\QQ$-point of $Zar(M)$ outside a union of proper algebraic subsets of $M$ defined over $\Q$ and of bounded degree.
\end{theorem}

\begin{remark}
Note that $\alpha$ can take only finitely many values, so the maximum and minimum in the above formula are indeed attained.
\end{remark}
\begin{remark}
The theorem implies that $\beta_\alpha(x)\geq \beta$ for almost every $x$ on $M$ if and only if $\beta_\alpha(y)\geq\beta$ for every $y$ on $M$, if and only if there exists a rational subspace $W$ such that $M$ is included in the (closed) algebraic subset of all $y\in\Hom(V,E)$ such that $\alpha(W\cap\ker y) \ge \beta(\dim W-\dim W\cap\ker y)$.
\end{remark}

\begin{proof}[Proof of Theorem~\ref{inhomogeneous}]
Since we can always reorder the $u_i^*$, $1\leq i\leq d$, we may assume without loss of generality that $0<\alpha_1\leq\dots\leq\alpha_d$.
Then, for $x$ in $\Hom(V,E)$, let
\[ m=\dim x(V) \quad\mbox{and}\quad n=\dim\ker x.\]
and define inductively a subset $I_x=\{i_{1,x},\dots,i_{n,x}\}$ in $\{1,\dots,d\}$ by
\[
\begin{array}{ll}
\mbox{for}\ j=1, & i_{1,x} = \min\{ i \ |\ u_i^*|_{\ker x}\neq 0\}\\
\forall j\geq 1, & i_{j+1,x} = \min\{ i \ |\ (u_{i_{1,x}}^*|_{\ker x},\dots,u_{i_{j,x}}^*|_{\ker x},u_{i}^*|_{\ker x}) \ \mbox{is linearly independent}\}.
\end{array}
\]
Now suppose $L_x$ is an element of $\GL(d,\R)$ with rows $L_{j,x}$ satisfying 
\begin{enumerate}
\item $\ker x=\bigcap_{1\leq j\leq m} \ker L_{j,x}$
\item $\forall j\in\{1,\dots,n\},\ L_{m+j,x}=u_{i_{j,x}}^*$
\end{enumerate}
Then $\beta_\alpha(x)$ is the minimal $\beta>0$ such that, for all $\eps>0$, for $Q$ large enough, the set of inequalities
\begin{equation}\label{inhomineq}
\left\{
\begin{array}{ll}
|L_{j,x}(v)| \leq Q^{-\beta-\eps} & \mbox{for}\ 1\leq i\leq m\\
|L_{m+j,x}(v)| \leq Q^{\alpha_{i_j}} & \mbox{for}\ 1\leq j\leq n
\end{array}
\right.
\end{equation}
has no non-zero solution $v$ in $\Delta$.

Let  $\beta$ be the right-hand side of $(\ref{formulaforexp})$. 
We want to show that for $\mu$-almost every $x$, for all $\eps>0$ and for $Q>0$ large enough, the inequalities (\ref{inhomineq}) have no non-zero solution in $\Delta$.
We prove this by induction on $d$, using our subspace theorem (i.e. Theorem \ref{subspace-main}).\\
\underline{$d=1$}: There is not much to prove in this case, since $\ker x$ is equal $\{0\}$ or $V$, for $\mu$-almost every $x$.\\
\underline{$d-1\to d$}: Assume the result already proven for $d-1$. 
The map $x\mapsto I_x$ is constant on a dense Zariski open subset of $Zar(M)$; we denote by $I=\{i_1,\dots,i_n\}$ its value on this dense open subset, and by $\widetilde{M}$ the set of all $L_x$, for $x$ in that subset and in $M$.
Since $\beta\geq\frac{\alpha(V)}{m}=\frac{\sum_{j=1}^n\alpha_{i_j}}{m}$, the solutions $v$ to (\ref{inhomineq}) satisfy
\[ \prod_{i=1}^d |L_{i,x}(v)| \leq Q^{-m\beta-m\eps+\sum_{j=1}^n\alpha_{i_j}}
\leq \|v\|^{-\eps'},\]
and therefore, by Theorem \ref{subspace-main}, there exist proper subspaces $V_1,\dots,V_\ell$ in $\Q^d$ containing all solutions except a finite number.
All we need to check is that on in each $V_i$, there can be only finitely many solutions. But this follows from the induction hypothesis, applied to $V'=V_i$, $\Delta'=\Delta\cap V_i$, $|\cdot| = |\cdot|_{V'}$ (the restriction of a quasi-norm to a subspace is comparable to a quasi-norm \cite[\S 4.1]{abrs2}), and to the submanifold $\{x|_{V_i} \ ;\ x\in \widetilde{M}\}$.
We leave it to the reader to check from the formula defining the exponent that $\beta'\leq\beta$.
\end{proof}

We can now easily derive our theorem about nilpotent groups.

\begin{proof}[Proof of Theorem~\ref{rothnilpotent}]
Let $V=\cF_{k,\g,\Q}$, $\Delta=\cF_{k,\g,\Q}(\Z)$ and $E=\g$.
For each $k$-tuple $\bg=(e^{X_1},\dots,e^{X_k})$ in $G^k$, we obtain an element $x_{\bg}$ in $\Hom(V,E)$ given by evaluation at $(X_1,\dots,X_k)$:
\[ x_\bg(\sr) = \sr(X_1,\dots,X_k).\]
The map $G^k \to \Hom(V,E)$, $\bg \mapsto x_\bg$ is a polynomial map with coefficients in $K$. In particular the Zariski closure of its image is defined over $\QQ$. When $\Gamma_{\bg}$ is relatively free,  Proposition~\ref{wordbracket} shows that
$\beta(\Gamma_{\bg}) = \beta(x_\bg)$,
where $\beta(x_{\bg})$ is the diophantine exponent with respect to the quasi-norm $|\cdot|$ on $V$ defined in $(\ref{quasinorm})$.
By Theorem~\ref{inhomogeneous},
\[ \beta(x_{\bg}) =  \max\left\{ \min_{\bh\in G^k}\frac{\alpha(W\cap\ker x_{\bh})}{\dim W-\dim W\cap\ker x_{\bh}} \ ;\ 
W\leq V \ \mbox{rational subspace}\right\}\] for Lebesgue almost every $\bg \in G^k$.
This shows that $\beta_k$ is well defined, and since $\alpha$ takes rational values, this formula shows that $\beta_k\in\Q$.
For each rational $W \leq V$, the set of all $\bh \in G^k$ such that $\alpha(W \cap \ker x_\bh) > \beta_k \dim (W/(W \cap \ker x_\bh)$ is a proper algebraic subset of $G^k$ defined by equations of bounded degree with coefficients in $K$. Their union forms a proper subset of $G^k$ by Lemma \ref{countable} and Theorem~\ref{inhomogeneous} implies that  $\beta(\Gamma_{\bg})=\beta_k$ for every $\bg$ outside this union. 
\end{proof}

\bigskip

\noindent \emph{Acknowledgements:} We are grateful to Barak Weiss for stimulating conversations about Harder-Narasimhan filtrations and to Mohammed Bardestani, Vesselin Dimitrov, Elon Lindenstrauss and P\'eter Varj\'u for interesting discussions around the subspace theorem.

\bibliographystyle{plain}
\bibliography{bibliography}

\end{document}